\pdfoutput=1
\documentclass[11pt,a4paper,american]{article}
\usepackage{babel}
\usepackage[T1]{fontenc}
\usepackage[utf8]{inputenc}
\usepackage{fullpage}
\usepackage[babel]{microtype}
\usepackage{csquotes}
\usepackage{authblk}
\usepackage{titlesec}
\usepackage{xcolor}
\usepackage[giveninits,sortcites,maxbibnames=10,style=ext-numeric]{biblatex}
\usepackage{hyperref}
\usepackage{enumitem}
\usepackage{amsmath}
\usepackage{mathtools}
\usepackage{amsthm}
\usepackage{thmtools}
\usepackage{dsfont}
\usepackage{amssymb}
\usepackage[noabbrev,capitalise,sort]{cleveref}
\usepackage{crossreftools}

\hypersetup{pdfencoding=auto,colorlinks=true,linkcolor=blue,citecolor=teal,bookmarksnumbered=true,hypertexnames=false,pdfauthor={Marco Pozza},pdftitle={Large Time Behavior of Solutions to Hamilton–Jacobi Equations on Networks},pdfkeywords={Hamilton–Jacobi equations;large time behavior;Aubry set;embedded networks},pdfsubject={MSC(2020): 35B40, 35R02, 49L25, 37J51}}

\titlespacing{\paragraph}{0pt}{1em}{.7em}

\declaretheorem[style=plain,parent=section,title=Theorem,refname={Theorem,Theorems}]{theo}
\declaretheorem[style=plain,sibling=theo,title=Proposition,refname={Proposition,Propositions}]{prop}
\declaretheorem[style=plain,sibling=theo,title=Corollary,refname={Corollary,Corollaries}]{cor}
\declaretheorem[style=plain,sibling=theo,title=Lemma,refname={Lemma,Lemmas}]{lem}
\declaretheorem[style=definition,sibling=theo,title=Definition,refname={Definition,Definitions}]{defin}
\declaretheorem[style=definition,sibling=theo,title=Example,refname={Example,Examples}]{ex}
\declaretheorem[style=remark,sibling=theo,title=Remark,refname={Remark,Remarks}]{rem}

\newlist{myenum}{enumerate}{1}
\setlist[myenum]{label=\textbf{\roman*)},ref=\textnormal{\textbf{(\roman*)}},font=\normalfont}
\newlist{mylist}{itemize}{1}
\setlist[mylist]{label=\textbullet,font=\normalfont}
\newlist{Hassum}{enumerate}{1}
\setlist[Hassum]{label=\textbf{(H\arabic*)},ref=\textnormal{\textbf{(H\arabic*)}},font=\normalfont}

\crefname{myenumi}{item}{items}
\crefname{Hassumi}{condition}{conditions}
\crefname{equation}{}{}

\DeclareFieldFormat{titlecase:title}{\MakeSentenceCase*{#1}}

\DeclareMathOperator{\diam}{diam}
\DeclareMathOperator{\interior}{int}

\newcommand{\Acal}{\mathcal A}
\newcommand{\wtAcal}{\widetilde\Acal}
\newcommand{\Ecal}{\mathcal E}
\newcommand{\Hcal}{\mathcal H}
\newcommand{\ovI}{\overline I}
\newcommand{\Ical}{\mathcal I}
\newcommand{\Jcal}{\mathcal J}
\newcommand{\Nds}{\mathds N}
\newcommand{\Qcal}{\mathcal Q}
\newcommand{\Rds}{\mathds R}
\newcommand{\Scal}{\mathcal S}
\newcommand{\ovT}{\overline T}
\newcommand{\Vbf}{\mathbf V}
\newcommand{\wtVbf}{\widetilde\Vbf}
\newcommand{\ova}{\overline a}
\newcommand{\ovs}{\overline s}
\newcommand{\oveta}{\overline\eta}
\newcommand{\wtgamma}{\widetilde\gamma}
\newcommand{\unphi}{\underline\phi}
\newcommand{\ovxi}{\overline\xi}
\newcommand{\myemail}[2][]{\textsuperscript{#1}\href{mailto:#2}{\texttt{#2}}}
\newcommand{\hrefc}{\textup{(}\hyperref[eq:globeik]{\ensuremath{\Hcal}\textup{J}\ensuremath{c}}\textup{)}}
\newcommand{\ghrefc}{\textup{(}\hyperref[eq:eikg]{\textup{HJ}\textsubscript{\ensuremath{\gamma}}\ensuremath{c}}\textup{)}}

\newcommand{\abscite}[2][]{\citeauthor{#2} \parentext{\citefield{#2}[journaltitle]{shortjournal}, \bibhyperref[#2]{\citeyear{#2}}}}

\title{Large Time Behavior of Solutions to Hamilton--Jacobi Equations on Networks}
\author{Marco Pozza}
\affil{Link Campus University, Rome, Italy. \textit{Email address:} \myemail{m.pozza@unilink.it}}
\date{}

\begin{document}

    \maketitle

    \begin{abstract}
        Starting from \abscite{NamahRoquejoffre99} and \abscite{Fathi98}, the large time asymptotic behavior of solutions to Hamilton--Jacobi equations has been extensively investigated by many authors, mostly on smooth compact manifolds and the flat torus. They all prove that such solutions converge to solutions to a corresponding static problem. We extend this study to the case where the ambient space is a network. The presence of a ``flux limiter'', that is the choice of appropriate constants on each vertex of the network necessary for the well-posedness of time-dependent problems on networks, enables a richer statement for the convergence compared to the classical setting. We indeed observe that solutions converge to subsolutions to a corresponding static problem depending on the value of the flux limiter. A finite time convergence is also established.
    \end{abstract}

    \paragraph{2020 Mathematics Subject Classification:} 35B40, 35R02, 49L25, 37J51.

    \paragraph{Keywords:} Hamilton--Jacobi equations, large time behavior, Aubry set, embedded networks.

    \section{Introduction}

    This paper is about the large time behavior of solutions to time-dependent Hamilton--Jacobi equations posed on networks.

    The subject has been extensively investigated on compact manifolds, in particular on the flat torus, first in~\cite{NamahRoquejoffre99,Fathi98} and subsequently in many other papers, among which we cite~\cite{BarlesSouganidis00,Roquejoffre01,DaviniSiconolfi06,BarlesIshiiMitake12}. They all show under suitable assumptions, that, given the solution $v$ to the time-dependent problem
    \begin{equation*}
        \left\{
        \begin{aligned}
            &\partial_t v+H(x,Dv)=0,\\
            &v(x,0)=\phi(x),
        \end{aligned}
        \right.
    \end{equation*}
    and letting $c$ be the critical value of the Hamiltonian $H$, the function $v(x,t)+ct$ uniformly converges, as $t$ positively diverges, to a solution $u$ of the critical equation
    \begin{equation*}
        H(x,Du)=c.
    \end{equation*}

    We consider a connected finite network $\Gamma$ embedded in $\Rds^N$ with vertices linked by regular simple curves $\gamma$ parametrized in $[0,1]$, called arcs of $\Gamma$. A Hamiltonian on $\Gamma$ is a collection of Hamiltonians $H_\gamma:[0,1]\times\Rds\to\Rds$ indexed by the arcs, depending on state and momentum variable, with the crucial feature that Hamiltonians associated to arcs possessing different support are totally unrelated.

    The equations we deal with are accordingly of the form
    \begin{equation}\label{eq:loctprob}
        \partial_t U(s,t)+H_\gamma(s,\partial_s U(s,t))=0,\qquad\text{on }(0,1)\times(0,\infty),
    \end{equation}
    on each arc $\gamma$, and a solution on $\Gamma$ is a continuous function $v:\Gamma\times[0,\infty)\to\Rds$ such that, for each arc $\gamma$, $v(\gamma(s),t)$ solves~\cref{eq:loctprob} in the viscosity sense and satisfies suitable additional conditions on the discontinuity interfaces
    \begin{equation*}
        \{(x,t),t\in[0,\infty)\}\qquad\text{with }x\in\Vbf,
    \end{equation*}
    where $\Vbf$ denotes the set of vertices. It has been established in~\cite{ImbertMonneau17}, in the case of junctions, and in~\cite{Siconolfi22}, for general networks, that to get existence and uniqueness of solutions, equations~\cref{eq:loctprob} must be coupled not only with a continuous initial datum at $t=0$, but also with a flux limiter, that is a choice of appropriate constants $c_x$ on each vertex $x$. In~\cite{LionsSouganidis17,Morfe20}, the time-dependent problem is studied on junctions, possibly multidimensional, with Kirchoff type Neumann conditions at vertices, without using flux limiters.\\
    In~\cite{Siconolfi22} flux limiters crucially appear in the conditions that a solution must satisfy on the interfaces and, among other things, bond from above the time derivatives of any subsolution on it. Even if an initial datum is fixed, solutions can change according to the choice of flux limiter, and they actually play a significant role in our analysis. Recently in~\cite{PozzaSiconolfi23} it has been given a Lax--Oleinik representation formula for the flux limited solutions to the evolutive problem, extending the result of~\cite{ImbertMonneauZidani12} where such a formula is given for the case of junctions.

    We prove here that, if $v$ is a solution to the time-dependent problem on $\Gamma$, then there is a unique constant $a$ depending on the flux limiter such that $v(x,t)+at$ possesses a uniform limit $u$, as $t$ positively diverges, $u$ such that $u\circ\gamma$ is a viscosity solution to the local problem
    \begin{equation*}
        H_\gamma(s,\partial_s U(s,t))=a,\qquad\text{on }(0,1),
    \end{equation*}
    for every arc $\gamma$. Under suitable assumptions the value $a$ coincides with the critical value of the Hamiltonian. This notion, as well as an extension of Weak KAM theory, has been first studied in~\cite{SiconolfiSorrentino18} in the framework of networks/graphs.\\
    A relevant peculiarity of the large time behavior problem on network is that the geometry of the network allows, under specific conditions, a finite time convergence. This will be useful for future applications and numerical analysis.

    We employ a dynamical approach to the problem exploiting the Lax--Oleinik formula given in~\cite{PozzaSiconolfi23}, the dynamic characterization of the solutions of the Eikonal equations and the properties of the Aubry set. To our knowledge there is no previous literature about the large time behavior of solutions to Hamilton--Jacobi equations on networks.

    The paper is organized as follows: in \cref{prelimsec} we fix some notation and conventions. In \cref{netsec} we provide some basic facts about networks and Hamiltonians defined on them, and give our main assumptions. In \cref{HJsec} we introduce Eikonal and time-dependent equations, together with some results relevant to our analysis. In \cref{convsec} we present the results of the asymptotic analysis. We distinguish three main cases according to the values of the flux limiter and the initial datum. In \cref{fixpointsec} we briefly discuss a characterization of the critical value involving the large time behavior.\\
    In the appendices we provide some auxiliary results. The reparametrization of curves on $\Gamma$ and their relationship with the representation formulas is the subject of \cref{repcurvesec}. \Cref{minact} is about the Lipschitz continuity of the minimal action functional. In \cref{curvecostsec} we provide the proof of a technical \namecref{liveoptcurve}.

    \paragraph{Acknowledgments.} The author acknowledges the support of the Italian Ministry of University and Research's PRIN 2022 grant \emph{``Stability in Hamiltonian dynamics and beyond''}. The author is a member of the INdAM research group GNAMPA.

    \section{Preliminaries}\label{prelimsec}

    We fix a dimension $N$ and $\Rds^N$ as ambient space. We also define
    \begin{equation*}
        \Rds^+=[0,\infty),\qquad\Qcal=(0,1)\times(0,\infty).
    \end{equation*}
    Notice that $\partial\Qcal=\{0\}\times[0,1]\cup\Rds^+\times\{0,1\}$.

    The scalar product between two elements $x$, $y$ of $\Rds^N$ is denoted with $x\cdot y$. We will use the notation $|\cdot|_2$ to indicate the Euclidean norm on $\Rds^N$.

    If $E\subset\Rds^N$ is a measurable set we denote with $|E|$ its \emph{Lebesgue measure}. We say that a property holds \emph{almost everywhere} (\emph{a.e.}\ for short) if it holds up to a set of measure zero.

    For all $f\in C(E)$, we define $\|f\|_\infty\coloneqq\sup\limits_{x\in E}|f(x)|$.

    Given two real numbers $a$ and $b$, we set
    \begin{equation*}
        a\wedge b\coloneqq\min\{a,b\},\qquad a\vee b\coloneqq\max\{a,b\}.
    \end{equation*}

    By curve we mean throughout the paper an \emph{absolutely continuous} curve with support contained in $\Rds^N$ or $\Rds$. We recall that a curve $\xi:[0,T]\to\Rds^N$ is \emph{closed} if $\xi(0)=\xi(T)$, \textit{simple} if $\xi(t)\ne\xi(t')$ whenever $t\in(0,T)$ and $t\ne t'\in[0,T]$. A point in the support of $\xi$ is called \emph{incident} to $\xi$.

    Let $\xi:[0,T]\to\Rds^N$ and $\xi':[0,T']\to\Rds^N$ be two curves such that $\xi(T)=\xi'(0)$. We define their \emph{concatenation} as the curve $\xi*\xi':[0,T+T']\to\Rds^N$ such that
    \begin{equation*}
        \xi*\xi'(t)\coloneqq\left\{
        \begin{aligned}
            &\xi(t),&&\text{if }t\in[0,T),\\
            &\xi'(t-T),&&\text{if }t\in\left[T,T+T'\right].
        \end{aligned}
        \right.
    \end{equation*}
    Notice that $*$ is an associative operation.

    Given an open set $\mathcal O$ and a continuous function $u:\overline{\mathcal O}\to\Rds$, we call \emph{supertangents} (resp.\ \emph{subtangents}) to $u$ at $x\in\mathcal O$ the viscosity test functions from above (resp.\ below). If needed, we take, without explicitly mentioning, $u$ and test function coinciding at $x$ and test function strictly greater (resp.\ less) than $u$ in a punctured neighborhood of $x$. We say that a subtangent $\varphi$ to $u$ at $x\in\partial\mathcal O$ is \emph{constrained to $\overline{\mathcal O}$} if $x$ is a minimizer of $u-\varphi$ in a neighborhood of $x$ intersected with $\overline{\mathcal O}$. See also \cite[Definition 3.4]{SiconolfiSorrentino18}.

    If $f$ is a locally Lipschitz continuous function we denote with $\partial f$ its Clarke's generalized gradient, see~\cite{Clarke90,ClarkeLedyaevSternWolenski98}. We point out that convex functions are locally Lipschitz continuous.

    \section{Networks}\label{netsec}

    \subsection{Basic Definitions}

    An \emph{embedded network}, or \emph{continuous graph}, is a subset $\Gamma\subset\Rds^N$ of the form
    \begin{equation*}
        \Gamma=\bigcup_{\gamma\in\Ecal}\gamma([0,1])\subset\Rds^N,
    \end{equation*}
    where $\Ecal$ is a finite collection of regular (i.e., $C^1$ with non-vanishing derivative) simple oriented curves, called \emph{arcs} of the network, that we assume, without any loss of generality, parameterized on $[0,1]$. Note that we are also assuming existence of one-sided derivatives at the endpoints 0 and 1. We stress out that a regular change of parameters does not affect our results.

    Observe that on the support of any arc $\gamma$, we also consider the inverse parametrization defined as
    \begin{equation*}
        \widetilde\gamma(s)\coloneqq\gamma(1-s),\qquad\text{for }s\in[0,1].
    \end{equation*}
    We call $\widetilde\gamma$ the \emph{inverse arc} of $\gamma$. We assume
    \begin{equation}\label{eq:nosovrap}
        \gamma((0,1))\cap\gamma'([0,1])=\emptyset,\qquad\text{whenever }\gamma'\ne\gamma,\widetilde\gamma.
    \end{equation}

    We call \emph{vertices} the initial and terminal points of the arcs, and denote by $\Vbf$ the sets of all such vertices. It follows from~\cref{eq:nosovrap} that vertices are the only points where arcs with different support intersect and, in particular,
    \begin{equation*}
        \gamma((0,1))\cap\Vbf=\emptyset,\qquad\text{for any }\gamma\in\Ecal.
    \end{equation*}

    We assume that the network is connected, namely given two vertices there is a finite concatenation of arcs linking them. A \emph{loop} is an arc with initial and final point coinciding. The unique restriction we require on the geometry of the network is
    \begin{enumerate}[label=\textbf{(A\arabic*)}]
        \item $\Ecal$ does not contain loops.
    \end{enumerate}
    This condition is due to the fact that in the known literature about time-dependent Hamilton--Jacobi equations on networks no loops are admitted, see, e.g., \cite{ImbertMonneau17,Siconolfi22,LionsSouganidis17}.

    For each $x\in\Vbf$, we define $\Gamma_x\coloneqq\{\gamma\in\Ecal:\gamma(1)=x\}$.

    The network $\Gamma$ inherits a geodesic distance, denoted with $d_\Gamma$, from the Euclidean metric of $\Rds^N$. It is clear that given $x$, $y$ in $\Gamma$ there is at least a geodesic linking them. The geodesic distance is in addition equivalent to the Euclidean one.

    We also consider a differential structure on the network by defining the \emph{tangent bundle} of $\Gamma$, $T\Gamma$ in symbols, as the set made up by the $(x,q)\in\Gamma\times\Rds^N$ with $q$ of the form
    \begin{equation*}
        q=\lambda\dot\gamma(s),\qquad\text{if $x=\gamma(s)$, $s\in[0,1]$, with $\lambda\in\Rds$}.
    \end{equation*}
    Note that $\dot\gamma(s)$ is univocally determined, up to a sign, if $x\in\Gamma\setminus\Vbf$ or in other words if $s\ne0,1$.

    We proceed recalling a result taken from~\cite{PozzaSiconolfi23} on this topic.

    \begin{lem}\label{invarccomp}
        For any given arc $\gamma$ and curve $\xi:[0,T]\to\gamma([0,1])$, the function
        \begin{equation*}
            \gamma^{-1}\circ\xi:[0,T]\to[0,1]
        \end{equation*}
        is absolutely continuous, and
        \begin{equation*}
            \frac d{dt}\gamma^{-1}\circ\xi(t)=\frac{\dot\gamma\left(\gamma^{-1}\circ\xi(t)\right)\cdot\dot\xi(t)}{|\dot\gamma(\gamma^{-1}\circ\xi(t))|_2^2},\qquad\text{for a.e. }t\in[0,T].
        \end{equation*}
    \end{lem}

    \begin{rem}\label{curveascomp}
        We notice that, for any given curve $\xi:[0,T]\to\Gamma$, there is an at most countable collection of open disjoint intervals $\{I_i\}$ with $\bigcup_i\ovI_i=[0,T]$ such that
        \begin{equation}\label{eq:curveascomp1}
            \xi\left(\ovI_i\right)\subseteq\gamma_i([0,1]),\qquad\text{for each index }i,
        \end{equation}
        where $\gamma_i$ is an arc of the network. Indeed if $\xi$ is constant this is trivial, otherwise we let
        \begin{equation*}
            E\coloneqq\{t\in[0,T]:\xi(t)\in\Vbf\}.
        \end{equation*}
        $E$ is closed therefore \cite[Theorem 1.9]{DeBarra03} yields
        \begin{equation*}
            [0,T]\setminus E=\bigcup_{i\in\Ical}I_i,
        \end{equation*}
        where $\Ical$ is an at most countable index set and $\{I_i\}_{i\in\Ical}$ is a collection of open disjoint intervals. Similarly
        \begin{equation*}
            [0,T]\setminus\left(\bigcup_{i\in\Ical}\ovI_i\right)=\bigcup_{j\in\Jcal}I_j,
        \end{equation*}
        where $\Jcal$ is an at most countable index set and $\{I_j\}_{\in\Jcal}$ is a collection of open disjoint intervals. It is apparent that $\bigcup\limits_{i\in\Ical\cup\Jcal}\ovI_i=[0,T]$ and~\cref{eq:curveascomp1} holds.\\
        We point out that each $\eta_i\coloneqq\gamma_i^{-1}\circ\xi|_{\ovI_i}$ is an absolutely continuous curve by \cref{invarccomp}.
    \end{rem}

    \subsection{Hamiltonians on \texorpdfstring{$\Gamma$}{Γ}}

    A Hamiltonian on $\Gamma$ is a collection of Hamiltonians $\Hcal\coloneqq\{H_\gamma\}_{\gamma\in\Ecal}$, where
    \begin{alignat*}{2}
        H_\gamma:[0,1]\times\Rds&&\:\longrightarrow\:&\Rds\\
        (s,\mu)&&\:\longmapsto\:&H_\gamma(s,\mu),
    \end{alignat*}
    satisfying
    \begin{equation}\label{eq:condcompH}
        H_{\wtgamma}(s,\mu)=H_\gamma(1-s,-\mu),\qquad\text{for any arc }\gamma.
    \end{equation}
    We emphasize that, apart the above compatibility condition, the Hamiltonians $H_\gamma$ are \emph{unrelated}.

    We require every $H_\gamma$ to be:
    \begin{Hassum}
        \item\label{condcont} continuous in both arguments;
        \item\label{condsuplin} $\lim\limits_{|\mu|\to\infty}\inf\limits_{s\in [0,1]}\dfrac{H_\gamma(s,\mu)}{|\mu|}=\infty$ for any $\gamma\in\Ecal$;
        \item\label{condconv} convex in $\mu$;
        \item\label{condqconv} strictly quasiconvex in $\mu$, which means that, for any $s\in[0,1]$, $\mu,\mu'\in\Rds$ and $\rho\in(0,1)$,
        \begin{equation*}
            H_\gamma\left(s,\rho\mu+(1-\rho)\mu'\right)<\max\left\{H_\gamma(s,\mu),H_\gamma\left(s,\mu'\right)\right\}.
        \end{equation*}
        If \cref{condconv} is satisfied, the above assumption is equivalent to
        \begin{equation}\label{eq:sublevelH}
            \interior\{\mu\in\Rds:H_\gamma(s,\mu)\le a\}=\{\mu\in\Rds:H_\gamma(s,\mu)<a\},\qquad\text{for any }a\in\Rds,\,s\in[0,1],
        \end{equation}
        where $\interior$ denotes the interior of a set.
    \end{Hassum}
    Note that if $H_\gamma$ is strictly convex in the momentum variable then it satisfies both~\ref{condconv} and~\ref{condqconv}.

    We define the support functions
    \begin{equation}\label{eq:sigmadef}
        \sigma_{\gamma,a}^+(s)\coloneqq\max\{\mu\in\Rds:H_\gamma(s,\mu)=a\},\qquad\sigma_{\gamma,a}^-(s)\coloneqq\min\{\mu\in\Rds:H_\gamma(s,\mu)=a\},
    \end{equation}
    with the assumption that when $\{\mu\in\Rds:H_\gamma(s,\mu)=a\}$ is empty $\sigma_{\gamma,a}^+(s)=-\infty$ and $\sigma_{\gamma,a}^-(s)=\infty$. It follows from~\cref{eq:condcompH} that
    \begin{equation}\label{eq:sigcomp}
        \sigma_{\widetilde\gamma,a}^+(s)=-\sigma_{\gamma,a}^-(1-s).
    \end{equation}
    Notice that $\{\mu\in\Rds:H_\gamma(s,\mu)=a\}$ is not empty if and only if $a\ge\min\limits_{\mu\in\Rds}H_\gamma(s,\mu)$, thus $\sigma_{\gamma,a}^+(s)\ne-\infty$ for any $s\in[0,1]$ if and only if $a\ge a_\gamma$, where
    \begin{equation*}
        a_\gamma\coloneqq\max_{s\in[0,1]}\min_{\mu\in\Rds}H_\gamma(s,\mu).
    \end{equation*}

    \begin{prop}\label{sigcont}\leavevmode
        \begin{myenum}
            \item\label{en:sigcont1} For each $\gamma\in\Ecal$ and $s\in[0,1]$ the function $a\mapsto\sigma_{\gamma,a}^+(s)$ is continuous and increasing in $\left[\min\limits_{\mu\in\Rds}H_\gamma(s,\mu),\infty\right)$;
            \item\label{en:sigcont2} for each $\gamma\in\Ecal$ and $a\ge a_\gamma$ the function $s\mapsto\sigma_{\gamma,a}^+(s)$ is continuous in $[0,1]$.
        \end{myenum}
    \end{prop}
    \begin{proof}
        \labelcref{en:sigcont1} follows from~\cref{eq:sublevelH}. For the proof of \labelcref{en:sigcont2} see \cite[Proposition 5.1]{SiconolfiSorrentino18}.
    \end{proof}

    Under assumptions \labelcref{condcont,condsuplin,condconv} it is natural to define, for any $\gamma\in\Ecal$, the \emph{Lagrangian} corresponding to $H_\gamma$ as
    \begin{equation*}
        L_\gamma(s,\lambda)\coloneqq\sup_{\mu\in\Rds}(\lambda\mu-H_\gamma(s,\mu)),
    \end{equation*}
    where the supremum is actually achieved thanks to \labelcref{condcont,condsuplin}. We have for each $\lambda\in\Rds$ and $s\in[0,1]$,
    \begin{equation}\label{eq:lagcomp}
        L_\gamma(s,\lambda)=L_{\widetilde\gamma}(1-s,-\lambda).
    \end{equation}
    We also derive that the Lagrangians $L_\gamma$ are superlinear. We will define later on a Lagrangian defined on the whole network, assuming suitable gluing conditions on the $L_\gamma$.

    It follows from the definition that, for any $\gamma\in\Ecal$,
    \begin{equation}\label{eq:loclagmin}
        \min_{s\in[0,1]}L_\gamma(s,0)=\min_{s\in[0,1]}\max_{\mu\in\Rds}(-H_\gamma(s,\mu))=-\max_{s\in[0,1]}\min_{\mu\in\Rds}H_\gamma(s,\mu)=-a_\gamma.
    \end{equation}

    \section{Hamilton--Jacobi Equations on Networks}\label{HJsec}

    \subsection{Eikonal HJ Equations}

    Here we are interested in equations of the form
    \begin{equation}\label{eq:globeik}\tag{\ensuremath{\Hcal}J\ensuremath{a}}
        \Hcal(x,Du)=a,\qquad\text{on }\Gamma,
    \end{equation}
    where $a\in\Rds$. This notation synthetically indicates the family of Hamilton--Jacobi equations
    \begin{equation}\label{eq:eikg}\tag{HJ\textsubscript{\ensuremath{\gamma}}\ensuremath{a}}
        H_\gamma(s,\partial_s U)=a,\qquad\text{on }[0,1],
    \end{equation}
    for $\gamma$ varying in $\Ecal$. This problem is thoroughly analyzed in~\cite{SiconolfiSorrentino18}, where the following definition of solution is given.

    \begin{defin}\label{defsol}
        We say that $w:\Gamma\to\Rds$ is a \emph{viscosity subsolution} to~\cref{eq:globeik} if
        \begin{myenum}
            \item it is continuous on $\Gamma$;
            \item $s\mapsto w(\gamma(s))$ is a viscosity subsolution to~\cref{eq:eikg} in $(0,1)$ for any $\gamma\in\Ecal$.
        \end{myenum}
        We say that $u:\Gamma\to\Rds$ is a \emph{viscosity solution} to~\cref{eq:globeik} if
        \begin{myenum}
            \item it is continuous;
            \item $s\mapsto u(\gamma(s))$ is a viscosity solution of~\cref{eq:eikg} in $(0,1)$ for any $\gamma\in\Ecal$;
            \item\label{stateconst} for every vertex $x$ there is at least one arc $\gamma\in\Gamma_x$ such that
            \begin{equation*}
                H_\gamma(1,\partial_s\varphi(1))\ge a
            \end{equation*}
            for any $\varphi$ that is a constrained $C^1$ subtangent to $s\mapsto u(\gamma(s))$ at 1.
        \end{myenum}
    \end{defin}

    In order to provide a representation formula for solution to~\cref{eq:globeik} we extend the support functions defined in~\cref{eq:sigmadef} to the tangent bundle $T\Gamma$ in the following sense: for any $a\in\Rds$ we set the map $\sigma_a:T\Gamma\to\Rds$ such that
    \begin{mylist}
        \item if $x=\gamma(s)$ for some $\gamma\in\Ecal$ and $s\in(0,1)$ then
        \begin{equation*}
            \sigma_a(x,q)\coloneqq\max\left\{\mu\frac{q\cdot\dot\gamma(s)}{|\dot\gamma(s)|_2^2}:\mu\in\Rds,\,H_\gamma(s,\mu)=a\right\}.
        \end{equation*}
        It is clear that if $\{\mu\in\Rds,\,H_\gamma(s,\mu)=a\}\ne\emptyset$ then, using the support functions defined in~\cref{eq:sigmadef},
        \begin{equation*}
            \sigma_a(x,q)=\left(\sigma_{\gamma,a}^+(s)\frac{q\cdot\dot\gamma(s)}{|\dot\gamma(s)|_2^2}\right)\vee\left(\sigma_{\gamma,a}^-(s)\frac{q\cdot\dot\gamma(s)}{|\dot\gamma(s)|_2^2}\right),
        \end{equation*}
        otherwise we assume that $\sigma_a(x,q)=-\infty$.
        \item If $x\in\Vbf$ and $q\ne0$ then
        \begin{equation*}
            \sigma_a(x,q)\coloneqq\min\max\left\{\mu\frac{q\cdot\dot\gamma(1)}{|\dot\gamma(1)|_2^2}:\mu\in\Rds,\,H_\gamma(1,\mu)=a\right\},
        \end{equation*}
        where the minimum is taken over the $\gamma\in\Gamma_x$ with $\dot\gamma(1)$ parallel to $q$. We assume that $\sigma_a(x,q)=-\infty$ whenever $\{\mu\in\Rds,\,H_\gamma(1,\mu)=a\}=\emptyset$ for any $\gamma\in\Gamma_x$ with $\dot\gamma(1)$ parallel to $q$.
        \item If $x\in\Vbf$ and $q=0$ then
        \begin{equation*}
            \sigma_a(x,q)\coloneqq0.
        \end{equation*}
    \end{mylist}
    Note that the case $x\in\Vbf$, $q\ne0$ is more involved because there is a problem to take into account, namely different arcs ending at $x$ could have parallel tangent vectors, in this case we should have
    \begin{equation*}
        q=\lambda_1\dot\gamma_1(1)=\lambda_2\dot\gamma_2(1),\qquad\text{for arcs $\gamma_1\ne\gamma_2$ and scalars $\lambda_1$, $\lambda_2$}.
    \end{equation*}
    We point out that thanks to~\cref{eq:sigcomp} $\sigma_a$ is a well-defined function in $T\Gamma$.

    We further define
    \begin{equation*}
        a_0\coloneqq\max_{\gamma\in\Ecal}a_\gamma
    \end{equation*}
    and, for $a\ge a_0$, the semidistance on $\Gamma$
    \begin{equation*}
        S_a(y,x)\coloneqq\min\left\{\int_0^T\sigma_a\left(\xi,\dot\xi\right)d\tau:\text{$\xi:[0,T]\to\Gamma$ is a curve from $y$ to $x$}\right\},
    \end{equation*}
    whose importance is highlighted by the next \namecref{subsolchar}.

    \begin{prop}\label{subsolchar}
        A continuous function $w:\Gamma\to\Rds$ is a subsolution to~\cref{eq:globeik} if and only if
        \begin{equation*}
            w(x)-w(y)\le S_a(y,x),\qquad\text{for any }x,y\in\Gamma.
        \end{equation*}
    \end{prop}
    \begin{proof}
        See~\cite{SiconolfiSorrentino18}.
    \end{proof}

    The \emph{critical value}, or \emph{Mañé critical value}, is defined as
    \begin{equation*}
        c\coloneqq\min\{a\ge a_0:\text{\cref{eq:globeik} admits subsolutions}\}
    \end{equation*}
    and, if $c>a_0$, it is the unique value such that~\hrefc{} (namely the equation~\cref{eq:globeik} with $a=c$) admits solutions in the sense of \cref{defsol}. This critical value is characterized by being the only $c>a_0$ such that, for all the closed curves $\xi:[0,T]\to\Gamma$,
    \begin{equation}\label{eq:critvalchar}
        \int_0^T\sigma_c\left(\xi,\dot\xi\right)d\tau\ge0
    \end{equation}
    and for at least one simple closed curve the inequality above is an identity. Hereafter $c$ will always denote the critical value for the Eikonal problem.\\
    If $c=a_0$ \cref{eq:critvalchar} still holds, however it is not guaranteed that~\cref{eq:critvalchar} is an identity for some simple closed curve nor that~\hrefc{} admits solutions. For that to be true we require, as in~\cite{SiconolfiSorrentino18}, the following condition:
    \begin{Hassum}[resume]
        \item for any $\gamma\in\Ecal$ with $a_\gamma=c=a_0$ the map $s\mapsto\min\limits_{p\in\Rds}H_\gamma(s,p)$ is constant in $[0,1]$.
    \end{Hassum}
    This, together with~\ref{condqconv}, further implies that if $\gamma\in\Ecal$ is such that $a_\gamma=c$ then
    \begin{equation}\label{eq:ceqa0sig}
        \sigma_{\gamma,a_\gamma}^+=\sigma_{\gamma,a_\gamma}^-,\qquad\text{on }[0,1],
    \end{equation}
    therefore~\cref{eq:critvalchar} is an identity for every $\xi\coloneqq\gamma*\wtgamma$ with $a_\gamma=c$.

    The following set, whose definition is deeply related to the critical value $c$, is crucial for our analysis.

    \begin{defin}
        We call \emph{Aubry set} on $\Gamma$, the closed set $\Acal_\Gamma$ made up of the support of the closed curves $\xi:[0,T]\to\Gamma$ with $\int_0^T\sigma_c\left(\xi,\dot\xi\right)d\tau=0$ and a.e.\ non-vanishing derivative. It follows from our previous discussion on the critical value that the Aubry set is nonempty.\\
        The Aubry set is partitioned into \emph{static classes}, defined as the equivalence classes with respect to the relation
        \begin{equation*}
            \left\{
            \begin{gathered}
                x,y\in\Gamma:\text{$x$ and $y$ are incident to a closed curve $\xi:[0,T]\to\Gamma$ with }\int_0^T\sigma_c\left(\xi,\dot\xi\right)d\tau=0\\
                \text{and a.e.\ non-vanishing derivative}
            \end{gathered}
            \right\}.
        \end{equation*}
    \end{defin}

    \begin{rem}\label{arcaubry}
        It is shown in~\cite{SiconolfiSorrentino18} that under our assumptions the Aubry set consists of the support of a collection of arcs. Moreover, exploiting~\cref{eq:ceqa0sig} it is easy to see that the support of the arcs $\gamma$ with $a_\gamma=c$ is contained in $\Acal_\Gamma$.
    \end{rem}

    Through static classes one can obtain a nice property of critical subsolutions, see \cite[Theorem 7.5]{SiconolfiSorrentino18}. We recall this result below for later mention.

    \begin{prop}\label{subsolstatclass}
        Let $w$ be a subsolution to~\hrefc{} and $\Gamma'$ be a static class of $\Acal_\Gamma$. Then
        \begin{equation*}
            w(x)=w(y)+S_c(y,x),\qquad\text{for every }x,y\in\Gamma'.
        \end{equation*}
    \end{prop}

    The main connection between Aubry set and critical (sub)solutions is given in the next \namecref{eiksol}, whose proof can be easily inferred from~\cite{SiconolfiSorrentino18}.

    \begin{theo}\label{eiksol}
        Let $\Gamma'$ be a closed subset of $\Gamma$, $w\in C(\Gamma)$ be a subsolution to~\hrefc{} and define
        \begin{equation*}
            u(x)\coloneqq\min_{y\in\Gamma'}(w(y)+S_c(y,x)),\qquad\text{for }x\in\Gamma.
        \end{equation*}
        Then $u$ is both a solution in $\Gamma\setminus(\Gamma'\setminus\Acal_\Gamma)$ and the maximal subsolution to~\hrefc{} agreeing with $w$ on $\Gamma'$. In particular, if $\Gamma'\subseteq\Acal_\Gamma$, then $u$ is a solution on the whole $\Gamma$.\\
        Furthermore, if $\Gamma'$ has nonempty intersection with all static classes of the Aubry set, then $u$ is the unique solution in $\Gamma\setminus(\Gamma'\setminus\Acal_\Gamma)$ agreeing with $w$ on $\Gamma'$.
    \end{theo}

    An analogue result holds for the supercritical case.

    \begin{theo}\label{supcritsol}
        \emph{\cite[Theorem 7.9(ii)]{SiconolfiSorrentino18}} Let $\Gamma'$ be a closed subset of $\Gamma$, $a>c$, $w\in C(\Gamma)$ be a subsolution to~\cref{eq:globeik} and define
        \begin{equation*}
            u(x)\coloneqq\min_{y\in\Gamma'}(w(y)+S_a(y,x)),\qquad\text{for }x\in\Gamma.
        \end{equation*}
        Then $u$ is both the unique solution in $\Gamma\setminus\Gamma'$ and the maximal subsolution to~\cref{eq:globeik} agreeing with $w$ on $\Gamma'$.
    \end{theo}

    The maximality properties stated in \cref{eiksol,supcritsol} can be generalized as follows:

    \begin{theo}\label{maxsubsol}
        Let $\Gamma'$ be a closed subset of $\Gamma$, $\phi:\Gamma'\to\Rds$ be any continuous function and define
        \begin{equation*}
            w(x)\coloneqq\min_{y\in\Gamma'}(\phi(y)+S_a(y,x)),\qquad\text{for }x\in\Gamma.
        \end{equation*}
        If $a\ge c$ then $w$ is the maximal subsolution to~\cref{eq:globeik} not exceeding $\phi$ on $\Gamma'$ and a solution in $\Gamma\setminus\Gamma'$.
    \end{theo}
    \begin{proof}
        This is a consequence of~\cite{SiconolfiSorrentino18}.
    \end{proof}

    We conclude this part on Eikonal equations with a result about the Lipschitz continuity of the subsolutions.

    \begin{prop}\label{subsollip}
        The subsolutions to~\cref{eq:globeik} are Lipschitz continuous for every $a\ge c$.
    \end{prop}
    \begin{proof}
        Fix $a\ge c$ and let $w$ be a subsolution to~\cref{eq:globeik}. By definition $w\circ\gamma$ is a subsolution to~\cref{eq:eikg} for any $\gamma\in\Ecal$, therefore \cite[Proposition 4.1]{BardiCapuzzo-Dolcetta97} yields that $w\circ\gamma$ is Lipschitz continuous. We further have by \cite[Lemma 3.1]{PozzaSiconolfi23} that $\gamma^{-1}$ is Lipschitz continuous with respect to the geodesic distance $d_\Gamma$ on $\Gamma$. The arcs $\gamma$ are finitely many, hence there is an $\ell>0$ such that both $w\circ\gamma$ and $\gamma^{-1}$ are $\ell$--Lipschitz continuous for all $\gamma\in\Ecal$. It follows that, given $z,z'\in\gamma([0,1])$ for some arc $\gamma$,
        \begin{equation}\label{eq:subsollip1}
            \left|w(z)-w\left(z'\right)\right|\le\ell\left|\gamma^{-1}(z)-\gamma^{-1}\left(z'\right)\right|\le\ell^2d_\Gamma\left(z,z'\right).
        \end{equation}
        Now let $x,y\in\Gamma$ and set a sequence $\{x_i\}_{i=0}^{n+1}$ such that $x_0=x$, $x_{n+1}=y$, $x_i$ is a vertex if $i\notin\{0,n+1\}$, $x_i$ and $x_{i+1}$ are incident to the same arc for every $i\in\{0,\dotsc,n\}$, and
        \begin{equation*}
            d_\Gamma(x,y)=\sum_{i=0}^n d_\Gamma(x_i,x_{i+1}).
        \end{equation*}
        We point out that such a sequence always exists since there is at least one geodesic linking $x$ and $y$. Exploiting~\cref{eq:subsollip1} we get
        \begin{equation*}
            |w(x)-w(y)|\le\sum_{i=0}^n|w(x_i)-w(x_{i+1})|\le\sum_{i=0}^n\ell^2d_\Gamma(x_i,x_{i+1})=\ell^2d_\Gamma(x,y).
        \end{equation*}
        $x$, $y$, $w$ and $a$ are arbitrary, thus this proves our claim.
    \end{proof}

    \subsection{Time-Dependent HJ Equations}\label{timeHJsec}

    In this \lcnamecref{timeHJsec} we focus on the following time-dependent problem on $\Gamma$:
    \begin{equation}\label{eq:globteik}\tag{\ensuremath{\Hcal}JE}
        \partial_t v(x,t)+\Hcal(x,Dv)=0,\qquad\text{on }\Gamma\times(0,\infty).
    \end{equation}
    This notation synthetically indicates the family (for $\gamma$ varying in $\Ecal$) of Hamilton--Jacobi equations
    \begin{equation}\label{eq:teikg}\tag{HJ\textsubscript{\ensuremath{\gamma}}E}
        \partial_t U(s,t)+H_\gamma(s,\partial_s U(s,t))=0,\qquad\text{on }\Qcal.
    \end{equation}

    Following~\cite{ImbertMonneau17} we call \emph{flux limiter} any function $x\mapsto c_x$ from $\Vbf$ to $\Rds$ satisfying
    \begin{equation*}
        c_x\ge\max_{\gamma\in\Gamma_x}a_\gamma,\qquad\text{for any }x\in\Vbf.
    \end{equation*}

    The definition of (sub/super)solutions to~\cref{eq:globteik} given in~\cite{PozzaSiconolfi23,Siconolfi22} is as follows:

    \begin{defin}\label{deftsubsol}
        We say that $w:\Gamma\times\Rds^+\to\Rds$ is a \emph{viscosity subsolution} to~\cref{eq:globteik} with flux limiter $c_x$ if
        \begin{myenum}
            \item it is continuous;
            \item\label{loctsubsol} $(s,t)\mapsto w(\gamma(s),t)$ is a viscosity subsolution to~\cref{eq:teikg} in $\Qcal$ for any $\gamma\in\Ecal$;
            \item\label{fluxcond} for any $T\in(0,\infty)$ and vertex $x$, if $\psi$ is a $C^1$ supertangent to $w(x,\cdot)$ at $T$ then $\partial_t\psi(T)\le-c_x$.
        \end{myenum}
    \end{defin}

    \begin{defin}\label{deftsupsol}
        We say that $v:\Gamma\times\Rds^+\to\Rds$ is a \emph{viscosity supersolution} to~\cref{eq:globteik} if
        \begin{myenum}
            \item it is continuous;
            \item\label{loctsupsol} $(s,t)\mapsto v(\gamma(s),t)$ is a viscosity supersolution to~\cref{eq:teikg} in $\Qcal$ for any $\gamma\in\Ecal$;
            \item\label{tstateconst} for every vertex $x$ and $T\in(0,\infty)$, if $\psi$ is a $C^1$ subtangent to $v(x,\cdot)$ at $T$ such that $\partial_t\psi(T)<-c_x$, then there is a $\gamma\in\Ecal$ such that $\gamma(1)=x$ and
            \begin{equation*}
                \partial_t\varphi(1,T)+H_\gamma(1,\partial_s\varphi(1,T))\ge0
            \end{equation*}
            for any $\varphi$ that is a constrained $C^1$ subtangent to $(s,t)\mapsto v(\gamma(s),t)$ at $(1,T)$. We stress out that this condition does not require the existence of constrained subtangents.
        \end{myenum}
        We say that $u:\Gamma\times\Rds^+\to\Rds$ is a \emph{viscosity solution} to~\cref{eq:globteik} if it is both a viscosity subsolution and supersolution.
    \end{defin}

    We also have a result concerning the existence of solutions.

    \begin{theo}\label{exunsolt}
        \emph{\cite[Theorem 6.7]{PozzaSiconolfi23}} Given an initial datum $\phi\in C(\Gamma)$ and a flux limiter $c_x$, \cref{eq:globteik} admits a unique solution $v$ with flux limiter $c_x$ such that $v(0,x)=\phi(x)$ for every $x\in\Gamma$.
    \end{theo}

    Hereafter we will usually assume that it is given a flux limiter $c_x$ for any $x\in\Vbf$. In view of the previous \namecref{exunsolt} we define, for every $t\in\Rds^+$, the nonlinear operator $\Scal(t)$ on $C(\Gamma)$ such that, for each $\phi\in C(\Gamma)$, $\Scal(t)\phi$ is the unique solution at the time $t$ to~\cref{eq:globteik} with initial datum $\phi$ and flux limiter $c_x$. The family of operators $\{\Scal(t)\}_{t\in\Rds^+}$ form a semigroup whose main properties are summarized below.

    \begin{prop}\label{Scalprop}\leavevmode
        \begin{myenum}
            \item\emph{(Semigroup property)} For any $t,t'\in\Rds^+$ we have $\Scal(t+t')=\Scal(t)\circ\Scal(t')$;
            \item\emph{(Monotonicity property)} for every $\phi_1,\phi_2\in C(\Gamma)$ such that $\phi_1\le\phi_2$ in $\Gamma$
            \begin{equation*}
                \Scal(t)\phi_1\le\Scal(t)\phi_2,\qquad\text{for any }t\in\Rds^+;
            \end{equation*}
            \item for any $\phi\in C(\Gamma)$, $t\in\Rds^+$ and $a\in\Rds$, $\Scal(t)(\phi+a)=\Scal(t)\phi+a$.
        \end{myenum}
    \end{prop}
    \begin{proof}
        The proof of this \namecref{Scalprop} is trivial in view of the formula~\cref{eq:vft} given below.
    \end{proof}

    We will provide a representation formula for solution to~\cref{eq:globteik} using a Lagrangian defined on the whole tangent bundle $T\Gamma$ of the network, namely the map $L:T\Gamma\to\Rds$ such that
    \begin{mylist}
        \item if $x=\gamma(s)$ for some $\gamma\in\Ecal$ and $s\in(0,1)$ then
        \begin{equation*}
            L(x,q)\coloneqq L_\gamma\left(s,\frac{q\cdot\dot\gamma(s)}{|\dot\gamma(s)|_2^2}\right);
        \end{equation*}
        \item if $x\in\Vbf$ and $q\ne0$ then
        \begin{equation*}
            L(x,q)\coloneqq\min L_\gamma\left(1,\frac{q\cdot\dot\gamma(1)}{|\dot\gamma(1)|_2^2}\right),
        \end{equation*}
        where the minimum is taken over the $\gamma\in\Gamma_x$ with $\dot\gamma(1)$ parallel to $q$;
        \item if $x\in\Vbf$ and $q=0$ then
        \begin{equation*}
            L(x,q)\coloneqq-c_x.
        \end{equation*}
    \end{mylist}
    We notice that thanks to~\cref{eq:lagcomp} $L$ is a well-defined function in $T\Gamma$.

    Following~\cite{PozzaSiconolfi23} the operators $\Scal(t)$ can then be represented through the integral formula
    \begin{equation}\label{eq:vft}
        (\Scal(t)\phi)(x)=\min\left\{\phi(\xi(0))+\int_0^t L\left(\xi,\dot\xi\right)d\tau:\text{$\xi$ is a curve with $\xi(t)=x$}\right\}.
    \end{equation}
    We stress out that there exists an optimal curve for $(\Scal(t)\phi)(x)$, see \cite[Theorem 5.2]{PozzaSiconolfi23}.

    \begin{theo}\label{loclipconttsol}
        Given $t_0>0$ there is a positive $\ell_{t_0}$ such that
        \begin{equation}\label{eq:loclipconttsol.1}
            \Gamma\times[t_0,\infty)\ni(x,t)\longmapsto(\Scal(t)\phi)(x)
        \end{equation}
        is $\ell_{t_0}$--Lipschitz continuous for any $\phi\in C(\Gamma)$.
    \end{theo}
    \begin{proof}
        We start noticing that
        \begin{equation}\label{eq:loclipconttsol1}
            (\Scal(t)\phi)(x)=\min_{y\in\Gamma}(\phi(y)+h_t(y,x)),\qquad\text{for any }(x,t)\in\Gamma\times\Rds^+,
        \end{equation}
        where $h_t(y,x)$ is defined by~\cref{eq:minact}. Let $(x,t),(x',t')\in\Gamma\times[t_0,\infty)$ and denote by $y$ an optimal point for $(\Scal(t)\phi)(x)$ in~\cref{eq:loclipconttsol1}. It is clear that
        \begin{equation*}
            d_\Gamma(y,x)\le\diam\Gamma\le\frac{\diam\Gamma}{t_0}t,\qquad\text{and}\qquad d_\Gamma\left(y,x'\right)\le\diam\Gamma\le\frac{\diam\Gamma}{t_0}t'.
        \end{equation*}
        We can then apply \cref{minactlip}, yielding that there is a constant $\ell_{t_0}$ such that
        \begin{equation*}
            \left(\Scal\left(t'\right)\phi\right)\left(x'\right)-(\Scal(t)\phi)(x)\le h_{t'}\left(y,x'\right)-h_t(y,x)\le\ell_{t_0}\left(d_\Gamma\left(x,x'\right)+\left|t-t'\right|\right).
        \end{equation*}
        Interchanging the roles of $(x,t)$ and $(x',t')$ in the previous analysis we get that the map in~\cref{eq:loclipconttsol.1} is $\ell_{t_0}$--Lipschitz continuous.
    \end{proof}

    \begin{cor}\label{unifconttsol}
        Given $\phi\in C(\Gamma)$ we have that
        \begin{equation}\label{eq:vftuc.1}
            \Gamma\times\Rds^+\ni(x,t)\longmapsto(\Scal(t)\phi)(x)
        \end{equation}
        is uniformly continuous.
    \end{cor}
    \begin{proof}
        It is shown in \cite[Proposition 6.6]{PozzaSiconolfi23} that~\cref{eq:vftuc.1} is continuous, and so is uniformly continuous in $\Gamma\times[0,1]$ by the Heine--Cantor Theorem. \Cref{loclipconttsol} implies the uniform continuity in $\Gamma\times[1,\infty)$ as well, concluding the proof.
    \end{proof}

    \section{Convergence to Steady States}\label{convsec}

    We start this \lcnamecref{convsec} by stating the main results of this article, the proofs will be given later.

    We first assume that the flux limiter is minimal, i.e.,
    \begin{equation}\label{eq:minfluxlim}
        c_x=\max_{\gamma\in\Gamma_x}a_\gamma,\qquad\text{for any }x\in\Vbf.
    \end{equation}
    In this case we retrieve the classic convergence result of~\cite{Fathi98} adapted to our setting.

    \begin{theo}\label{protoltb}
        Given the flux limiter~\cref{eq:minfluxlim} and a $\phi\in C(\Gamma)$, we define
        \begin{equation}\label{eq:protoltb.1}
            u(x)\coloneqq\min_{y\in\Acal_\Gamma}\left(\min_{z\in\Gamma}(\phi(z)+S_c(z,y))+S_c(y,x)\right),\qquad\text{for }x\in\Gamma.
        \end{equation}
        Then $\Scal(t)\phi+ct$ uniformly converges, as $t$ goes to $\infty$, to $u$.
    \end{theo}

    Note that by \cref{eiksol,maxsubsol} $u$ in~\cref{eq:protoltb.1} is the unique solution in $\Gamma$ to~\hrefc{} agreeing with
    \begin{equation}\label{eq:protoltb.2}
        w(x)\coloneqq\min\limits_{y\in\Gamma}(\phi(y)+S_c(y,x)),\qquad\text{for }x\in\Gamma,
    \end{equation}
    on $\Acal_\Gamma$.

    Instead of \cref{protoltb}, we will prove a generalization of it. We proceed assuming, more generally, that
    \begin{equation}\label{eq:fluxcrit}
        c_x\le c,\qquad\text{for any }x\in\Vbf,
    \end{equation}
    and define
    \begin{equation}\label{eq:extvert}
        \wtVbf\coloneqq\{x\in\Vbf\setminus\Acal_\Gamma:c_x=c\}.
    \end{equation}
    We will see later that, roughly speaking, the optimal curves of~\cref{eq:vft} do not distinguish $\wtVbf$ from the Aubry set as the time diverges, so that, defining the \emph{extended Aubry set}
    \begin{equation*}
        \wtAcal_\Gamma\coloneqq\Acal_\Gamma\cup\wtVbf,
    \end{equation*}
    we can obtain the following \namecref{ltb}:

    \begin{theo}\label{ltb}
        Given a flux limiter $c_x$ satisfying~\cref{eq:fluxcrit} and $\phi\in C(\Gamma)$, the function $\Scal(t)\phi+ct$ uniformly converges, as $t$ goes to $\infty$, to
        \begin{equation}\label{eq:ltb.1}
            u(x)\coloneqq\min_{y\in\wtAcal_\Gamma}\left(\min_{z\in\Gamma}(\phi(z)+S_c(z,y))+S_c(y,x)\right),\qquad\text{for }x\in\Gamma.
        \end{equation}
    \end{theo}

    If $w$ is defined by~\cref{eq:protoltb.2}, the limit function given by the above result has the properties of being, thanks to \cref{eiksol,maxsubsol}, the unique solution to~\hrefc{} in $\Gamma\setminus\wtVbf$ agreeing with $w$ on $\wtAcal_\Gamma$, as well as the maximal subsolution in $\Gamma$ equaling $w$ on $\wtAcal_\Gamma$. We stress out that, for the large time behavior, the extended Aubry set on networks plays the same role as the Aubry set on compact manifolds.

    We now assume that there is some $y\in\Vbf$ such that $c_y>c$. Let $x$ be an arbitrary point of $\Gamma$ and $t>T$ two positive times, we consider a curve $\xi:[0,t]\to\Gamma$ from $y$ to $x$ satisfying $\xi(\tau)\equiv y$ in $[0,t-T]$, then, for any $\phi\in C(\Gamma)$, we get
    \begin{equation*}
        (\Scal(t)\phi)(x)+ct\le\phi(y)+\int_0^t\left(L\left(\xi,\dot\xi\right)+c\right)d\tau=\phi(y)+\int_{t-T}^t\left(L\left(\xi,\dot\xi\right)+c\right)d\tau+(c-c_y)(t-T).
    \end{equation*}
    Since $c-c_y<0$, this shows that
    \begin{equation*}
        \lim_{t\to\infty}(\Scal(t)\phi)(x)+ct=-\infty,\qquad\text{for any }\phi\in C(\Gamma).
    \end{equation*}
    We can, however, retrieve convergence as $t\to\infty$, but in contrast to the previous cases, not anymore to critical sub/solutions, but instead to suitable supercritical subsolutions, as shown in the next \namecref{supltb}.

    \begin{theo}\label{supltb}
        Given a flux limiter $c_x$ such that $c_x>c$ for some $x\in\Vbf$ and a $\phi\in C(\Gamma)$ we define
        \begin{equation}\label{eq:supltb.1}
            \ova\coloneqq\max_{x\in\Vbf}c_x>c,\qquad\Vbf_{\ova}\coloneqq\{x\in\Vbf:c_x=\ova\}
        \end{equation}
        and
        \begin{equation*}
            u(x)\coloneqq\min_{y\in\Vbf_{\ova}}\left(\min_{z\in\Gamma}(\phi(z)+S_{\ova}(z,y))+S_{\ova}(y,x)\right),\qquad\text{for }x\in\Gamma.
        \end{equation*}
        There exists a time $T$, depending on $\phi$ and $\ova$, such that $\Scal(t)\phi+\ova t\equiv u$ on $\Gamma$ for any $t\ge T$.
    \end{theo}

    Similarly to what happens in the previous cases, if we set
    \begin{equation*}
        w(x)\coloneqq\min_{z\in\Gamma}(\phi(y)+S_{\ova}(y,x)),\qquad\text{for }x\in\Gamma,
    \end{equation*}
    \cref{supcritsol,maxsubsol} yield that the limit function is the unique solution in $\Gamma\setminus\Vbf_{\ova}$ and the maximal subsolution in $\Gamma$ to~\cref{eq:globeik}, coinciding with $w$ on $\Vbf_{\ova}$.

    Under suitable assumptions, we can establish a finite time convergence result also when~\cref{eq:fluxcrit} holds.

    \begin{theo}\label{ltbfiniteweak}
        Assume that $c>a_0$, let $c_x$ be a flux limiter satisfying~\cref{eq:fluxcrit} and $u$ be defined by~\cref{eq:ltb.1}. If in each static class of $\Acal_\Gamma$ there is a vertex $x$ with $c_x=c$, then, for any $\phi\in C(\Gamma)$, there is a constant $T_\phi$ depending on $\phi$ such that $\Scal(t)\phi+ct\equiv u$ on $\Gamma$ whenever $t\ge T_\phi$.
    \end{theo}

    Finite time convergence can be also achieved assuming the initial datum to be a subsolution to~\hrefc{}.

    \begin{prop}\label{subsolltb}
        Given a flux limiter $c_x$ satisfying~\cref{eq:fluxcrit} and a subsolution $w$ to~\hrefc{}, we define
        \begin{equation*}
            u(x)\coloneqq\min_{y\in\wtAcal_\Gamma}(w(y)+S_c(y,x)),\qquad\text{for }x\in\Gamma.
        \end{equation*}
        Then there is a time $T_w$ depending on $w$ such that
        \begin{equation}\label{eq:subsolltb.1}
            (\Scal(t)w)(x)+ct=u(x),\qquad\text{for any }x\in\Gamma,\,t\ge T_w.
        \end{equation}
    \end{prop}

    It has been proved in~\cite{SiconolfiSorrentino18} that the trace on $\Vbf$ of any critical solution on $\Gamma$ is solution to an appropriate discrete functional equation. Conversely, a solution of the same discrete equation on $\Vbf$ can be uniquely extended to a critical solution on $\Gamma$. In this vein, it is reasonable to assume that the limit in~\cref{eq:protoltb.1} is not affected by the values of $\Scal(t)\phi+ct$ outside on the vertices, e.g., the values of the initial datum $\phi$ on $\Gamma\setminus\Vbf$. However this is not generally true, as can be seen in the next \namecref{asympnofocusV}.

    \begin{ex}\label{asympnofocusV}
        Let $\Gamma$ be a network with only two vertices and a single arc $\gamma$ connecting them. If we define
        \begin{equation*}
            H_\gamma(\mu)=\mu^2,\qquad\text{for }\mu\in\Rds,
        \end{equation*}
        it is easy to check that the critical value of the Eikonal problem defined by this Hamiltonian is 0. Moreover we have that $\sigma_{\gamma,0}^+\equiv0\equiv\sigma_{\gamma,0}^-$, thus $S_0(y,x)=0$ for any $x,y\in\Gamma$. Then, given $\phi\in C(\Gamma)$, \cref{protoltb} yields that, as $t$ goes to $\infty$, $\Scal(t)\phi$ uniformly converges to the minimum of $\phi$, independently of where this value is attained.
    \end{ex}

    \subsection{Convergence in Finite Time}\label{convfintime}

    The purpose of this \lcnamecref{convfintime} is to provide the proofs of \cref{subsolltb,ltbfiniteweak}, using some auxiliary results.

    \begin{prop}\label{weakltb}
        Let $c_x$ be a flux limiter satisfying~\cref{eq:fluxcrit} and $u$ be a solution to~\hrefc{} in $\Gamma\setminus\wtVbf$, where $\wtVbf$ is defined as in~\cref{eq:extvert}, then $\Scal(t)u=u-ct$ on $\Gamma\times\Rds^+$.
    \end{prop}
    \begin{proof}
        First we fix $\gamma\in\Ecal$ and let $\varphi$ be a $C^1$ supertangent to $u(\gamma(s))-ct$ at some $(s^*,t^*)\in\Qcal$. It is apparent that $s\mapsto\varphi(s,t^*)$ is a supertangent to $u\circ\gamma$ at $s^*$, therefore, since $u$ is a subsolution to~\hrefc{}, we have
        \begin{equation}\label{eq:weakltb1}
            H_\gamma(s^*,\partial_s\varphi(s^*,t^*))\le c.
        \end{equation}
        Next we notice that for each $h>0$ small enough
        \begin{equation*}
            \frac{\varphi(s^*,t^*-h)-\varphi(s^*,t^*)}{-h}\le\frac{u(\gamma(s^*))-c(t^*-h)-u(\gamma(s^*))+ct^*}{-h}=-c,
        \end{equation*}
        which shows, together with~\cref{eq:weakltb1}, that
        \begin{equation*}
            \partial_t\varphi(s^*,t^*)+H_\gamma(s^*,\partial_s\varphi(s^*,t^*))\le0.
        \end{equation*}
        This fact, taking into account that $\varphi$, $\gamma$ and $(s^*,t^*)$ are arbitrary, yields that $u-ct$ satisfies \cref{loctsubsol} in \cref{deftsubsol} of subsolution to the time-dependent problem. Similarly, we can show that $u-ct$ satisfies~\ref{loctsupsol} in \cref{deftsupsol} of supersolution. Moreover, it follows from~\cref{eq:fluxcrit} that~\ref{fluxcond} in \cref{deftsubsol} holds true for $u-ct$. Finally, by definition, whenever $x\in\Vbf\setminus\wtVbf$, i.e., $c_x<c$, \labelcref{stateconst} in \cref{defsol} of solution to the stationary equation holds true, therefore $u-ct$ also satisfies~\ref{tstateconst} in \cref{deftsupsol}. This yields that $u-ct$ is a solution to~\cref{eq:globteik}, which proves our claim.
    \end{proof}

    Fixed $\phi\in C(\Gamma)$, let $u$ be as in~\cref{eq:ltb.1} and set $\alpha\coloneqq0\vee\max\limits_{x\in\Gamma}(\phi(x)-u(x))$. Then \cref{Scalprop,weakltb} yield that
    \begin{equation}\label{eq:asbound}
        (\Scal(t)\phi)(x)\le u(x)-ct+\alpha,\qquad\text{for any }(x,t)\in\Gamma\times\Rds^+.
    \end{equation}
    The next \namecref{liveoptcurve}, whose proof is given in \cref{curvecostsec}, is a consequence of this inequality.

    \begin{lem}\label{liveoptcurve}
        Given $\phi\in C(\Gamma)$ and a flux limiter $c_x$ satisfying~\cref{eq:fluxcrit}, there is $T_\phi>0$ depending only on $\phi$ such that, for any $x\in\Gamma$, $t\ge T_\phi$ and any optimal curve $\xi$ for $(\Scal(t)\phi)(x)$, $\xi([0,t])\cap\wtAcal_\Gamma\ne\emptyset$.
    \end{lem}

    \Cref{liveoptcurve} shows that an optimal curve can stay outside the extended Aubry set only for a finite time. This is the key point for proving the finite time convergence results.

    \begin{proof}[Proof of \cref{subsolltb}]
        By \cref{liveoptcurve} there is a constant $T_w$ such that, fixed $x\in\Gamma$, $t\ge T_w$ and an optimal curve $\xi$ for $(\Scal(t)w)(x)$, there is a $t'\in[0,t]$ such that $\xi(t')\in\wtAcal_\Gamma$. Then it follows from \cref{laglbound,subsolchar} that
        \begin{align*}
            (\Scal(t)w)(x)+ct=&\,w(\xi(0))+\int_0^{t'}\left(L\left(\xi,\dot\xi\right)+c\right)d\tau+\int_{t'}^t\left(L\left(\xi,\dot\xi\right)+c\right)d\tau\\
            \ge&\,w(\xi(0))+S_c(\xi(0),\xi(t'))+S_c(\xi(t'),x)\ge w(\xi(t'))+S_c(\xi(t'),x)\\
            \ge&\,\min_{y\in\wtAcal_\Gamma}(w(y)+S_c(y,x))=u(x).
        \end{align*}
        Since the pair $(x,t)$ is arbitrary, this shows that
        \begin{equation}\label{eq:subsolltb1}
            (\Scal(t)w)(x)+ct\ge u(x),\qquad\text{for any }x\in\Gamma,t\ge T_w.
        \end{equation}
        Finally, since $w\le u$, \cref{eq:asbound} yields
        \begin{equation*}
            (\Scal(t)w)(x)+ct\le u(x),\qquad\text{for any }(x,t)\in\Gamma\times\Rds^+,
        \end{equation*}
        which, together with~\cref{eq:subsolltb1}, proves~\cref{eq:subsolltb.1}.
    \end{proof}

    We conclude this \lcnamecref{convfintime} proving a more general version of \cref{ltbfiniteweak} using Lagrangian parametrizations, see \cref{lagpardef}.

    \begin{theo}\label{ltbfinite}
        Assume that every $\gamma\in\Ecal$ admits a $c$--Lagrangian reparametrization, let $c_x$ be a flux limiter satisfying~\cref{eq:fluxcrit} and $u$ be defined by~\cref{eq:ltb.1}. If in each static class of $\Acal_\Gamma$ there is a vertex $x$ with $c_x=c$, then, for any $\phi\in C(\Gamma)$, there is a constant $T_\phi$ depending on $\phi$ and such that $\Scal(t)\phi+ct\equiv u$ on $\Gamma$ whenever $t\ge T_\phi$, where $u$ is defined by~\cref{eq:ltb.1}
    \end{theo}

    Notice that, by \cref{arcaubry,timeoflagpar}, this \namecref{ltbfinite} depends on the dynamical properties of the Aubry set as well as on the flux limiter. In particular, if $c>a_0$, then each $\gamma\in\Ecal$ has a $c$--Lagrangian reparametrization, i.e., \cref{ltbfinite} implies \cref{ltbfiniteweak}.

    \begin{proof}
        We preliminarily observe that if $w$ is defined by~\cref{eq:protoltb.2}, then \cref{maxsubsol,Scalprop} yield
        \begin{equation*}
            (\Scal(t)w)(x)+ct\le(\Scal(t)\phi)(x)+ct,\qquad\text{for any }(x,t)\in\Gamma\times\Rds^+.
        \end{equation*}
        It follows from \cref{subsolltb} that, for any $x\in\Gamma$
        \begin{equation}\label{eq:ltbfinite1}
            (\Scal(t)\phi)(x)+ct\ge u(x),\qquad\text{for any }x\in\Gamma,\,t\ge T,
        \end{equation}
        where $T$ is a constant depending on $\phi$. Next we fix $x\in\Gamma$ and let $y\in\wtAcal_\Gamma$ and $z\in\Gamma$ be such that
        \begin{equation}\label{eq:ltbfinite2}
            u(x)=w(y)+S_c(y,x)=\phi(z)+S_c(z,y)+S_c(y,x)
        \end{equation}
        and let $\xi_1:[0,T_1]\to\Gamma$ and $\xi_2:[0,T_2]\to\Gamma$ be two simple curves optimal for $S_c(z,y)$ and $S_c(y,x)$, respectively. Exploiting \cref{subsolstatclass} and our assumptions, we assume without loss of generality that $y\in\Vbf$ and $c_y=c$. Moreover, we observe that since every $\gamma\in\Ecal$ admits a $c$--Lagrangian reparametrization, then, by \cref{curveascomp,repconstspeed,repsigma}, we can further assume that $\xi_1$ and $\xi_2$ have a $c$--Lagrangian parametrization. If we define, for any $t\ge T_1+T_2$,
        \begin{equation*}
            \xi_t(r)\coloneqq\left\{
            \begin{aligned}
                &\xi_1(r),&&\text{if }r\in[0,T_1],\\
                &y,&&\text{if }r\in(T_1,t-T_2),\\
                &\xi_2(r-(t-T_2)),&&\text{if }r\in[t-T_2,t],
            \end{aligned}
            \right.
        \end{equation*}
        it is then apparent that
        \begin{equation*}
            \int_0^t\left(L\left(\xi_t,\dot\xi_t\right)+c\right)d\tau=S_c(z,y)+S_c(y,x),\qquad\text{for any }t\ge T_1+T_2,
        \end{equation*}
        thus in view of~\cref{eq:ltbfinite2} we get
        \begin{equation}\label{eq:ltbfinite3}
            (\Scal(t)\phi)(x)+ct\le u(x),\qquad\text{for any }t\ge T_1+T_2.
        \end{equation}
        Finally \cref{timeoflagpar} yields that
        \begin{equation*}
            T_\gamma(c)\coloneqq\left\{t>0:\text{$\gamma$ has a $c$--Lagrangian reparametrization defined in $[0,t]$}\right\}
        \end{equation*}
        is a compact interval for all $\gamma\in\Ecal$, i.e.,
        \begin{equation*}
            T_\gamma(c)=\left[\underline T_\gamma(c),\ovT_\gamma(c)\right];
        \end{equation*}
        hence it is simple to check that the constant $T_c\coloneqq\sum_\gamma\ovT_\gamma(c)$, which depends only on $c$, is bigger than both $T_1$ and $T_2$. Then, by~\cref{eq:ltbfinite3} and the fact that $x$ is arbitrary,
        \begin{equation}\label{eq:ltbfinite4}
            (\Scal(t)\phi)(x)+ct\le u(x),\qquad\text{for any }x\in\Gamma,\,t\ge 2T_c,
        \end{equation}
        therefore, setting $T_\phi\coloneqq2T_c\vee T$, \cref{eq:ltbfinite1,eq:ltbfinite4} conclude the proof.
    \end{proof}

    \subsection{Convergence in the General Case}

    Here we will prove \cref{ltb}. In order to do so we introduce a uniform limits set and analyze the dynamical properties of the extended Aubry set. The analysis goes along the same lines as the one performed in~\cite{DaviniSiconolfi06}.

    First we observe that the vertices in $\wtVbf$, where $\wtVbf$ is defined as in~\cref{eq:extvert}, and the static classes of $\Acal_\Gamma$ form a partition of $\wtAcal_\Gamma$, whose elements we will henceforth refer to as \emph{static classes} of the extended Aubry set. Noticing that the static classes outside $\Acal_\Gamma$ are singletons, we get the following extension of \cref{subsolstatclass}.

    \begin{lem}\label{subsolextstatclass}
        If $\Gamma'$ is a static class of $\wtAcal_\Gamma$ and $w$ is a subsolution to~\hrefc{} then
        \begin{equation*}
            w(x)=w(y)+S_c(y,x),\qquad\text{for any }x,y\in\Gamma'.
        \end{equation*}
    \end{lem}

    The asymptotic character of our analysis require the use of a special class of curves.

    \begin{defin}
        We call \emph{static curve} any curve $\zeta:\Rds\to\Gamma$ with support contained in the extended Aubry set and such that
        \begin{equation*}
            \int_{t_1}^{t_2}\left(L\left(\zeta,\dot\zeta\right)+c\right)d\tau=S_c(\zeta(t_1),\zeta(t_2)),\qquad\text{for any }t_2\ge t_1.
        \end{equation*}
        As a consequence of \cref{laglbound} we have that $\zeta$ has $c$--Lagrangian parametrization.
    \end{defin}

    On smooth manifolds it is known, see, e.g., \cite{DaviniSiconolfi06}, that through any point of the Aubry set passes a static curve. On networks we further have that on each static class there is a static curve.

    \begin{prop}\label{critcurveperiod}
        Each static class of the extended Aubry set contains a periodic static curve.
    \end{prop}
    \begin{proof}
        Given a static class $\Gamma'\subseteq\Acal_\Gamma$ there is a closed curve $\xi:[0,T]\to\Gamma$ with a.e.\ non-vanishing derivative and support contained in it such that
        \begin{equation*}
            \int_0^T\sigma_c\left(\xi,\dot\xi\right)d\tau=S_c(\xi(0),\xi(0))=0.
        \end{equation*}
        If $c$ is admissible for $\xi$ (see \cref{admisdef}) we can assume, thanks to \cref{repsigma,timeoflagpar}, that $\xi$ has $c$--Lagrangian parametrization. We then have that
        \begin{equation*}
            \zeta(t)\coloneqq\xi\left(t-\left\lfloor\frac tT\right\rfloor T\right),\qquad\text{for }t\in\Rds,
        \end{equation*}
        is a periodic static curve contained in $\Gamma'$.\\
        If $c$ is not admissible or $\Gamma'\subseteq\wtVbf$ there is an $x\in\Gamma'$ such that $L(x,0)=-c$, thus the curve $\zeta:\Rds\to\{x\}$ is a periodic static curve contained in $\Gamma'$.
    \end{proof}

    Now assume that~\cref{eq:fluxcrit} holds, then, given a $\phi\in C(\Gamma)$, $u$ defined by~\cref{eq:ltb.1} and $w$ as in~\cref{eq:protoltb.2}, we have by~\cref{eq:asbound,Scalprop} that
    \begin{equation}\label{eq:Scal2bound}
        (\Scal(t)w)(x)+ct\le(\Scal(t)\phi)(x)+ct\le u(x)+\alpha,\qquad\text{for any }(x,t)\in\Gamma\times\Rds^+.
    \end{equation}
    Thanks to this and \cref{unifconttsol}, the Arzelà--Ascoli Theorem yields that for any positive diverging sequence $\{t_n\}_{n\in\Nds}$, up to subsequences, $\Scal(t_n)\phi+ct_n$ converges uniformly to some continuous function $f$. We denote with $\omega_\Scal(\phi)$ the set made up by the uniform limits of $\Scal(t)\phi+ct$. We point out that by~\cref{eq:Scal2bound,subsolltb}
    \begin{equation}\label{eq:omegalbound}
        f(x)\ge u(x),\qquad\text{for any }f\in\omega_\Scal(\phi),x\in\Gamma.
    \end{equation}
    We further set the semilimit
    \begin{equation}\label{eq:limsuptsol}
        \unphi(x)\coloneqq\sup\left\{\limsup_{n\to\infty}(\Scal(t_n)\phi)(x_n)+ct_n\right\},
    \end{equation}
    where the supremum is taken over the sequences $\{x_n\}_{n\in\Nds}$ converging to $x$ and the positive diverging sequences $\{t_n\}_{n\in\Nds}$. In view of the uniform continuity of $(\Scal(t)\phi)(x)$ proved in \cref{unifconttsol}, $\unphi$ is continuous and the sequences $\{x_n\}$ may be chosen identically equal to $x$. It follows that
    \begin{equation}\label{eq:unphisup}
        \unphi(x)=\sup\{f(x):f\in\omega_\Scal(\phi)\}.
    \end{equation}

    \begin{prop}\label{unphisubsol}
        Given $\phi\in C(\Gamma)$ and a flux limiter satisfying~\cref{eq:fluxcrit}, let $\unphi$ be as in~\cref{eq:limsuptsol}. Then $\unphi$ is a subsolution to~\hrefc{}.
    \end{prop}
    \begin{proof}
        We have seen above that $\unphi$ is continuous, thereby to prove our claim it is enough to show that $\unphi\circ\gamma$ is a subsolution to~\ghrefc{} for any $\gamma\in\Ecal$, see \cref{defsol}.\\
        We start fixing a $\gamma\in\Ecal$, a supertangent $\varphi$ to $\unphi\circ\gamma$ at a point $\ovs\in(0,1)$, a $\delta>0$ and a sequence $\{t_n\}_{n\in\Nds}$ such that $t_n>\delta$ for all $n\in\Nds$ and
        \begin{equation*}
            \lim_{n\to\infty}(\Scal(t_n)\phi)(\gamma(\ovs))+ct_n=\unphi(\gamma(\ovs)).
        \end{equation*}
        We further set for each $n\in\Nds$
        \begin{align*}
            v_n:[0,1]\times[-\delta,\delta]&\longrightarrow\Rds,\\
            (s,t)&\longmapsto(\Scal(t_n+t)\phi)(\gamma(s))+c(t_n+t),
        \end{align*}
        then~\cref{eq:Scal2bound}, \cref{unifconttsol} and the Arzelà--Ascoli Theorem yield that, up to subsequences, $\{v_n\}$ uniformly converges to a $v\in C([0,1]\times[-\delta,\delta])$. It is clear that each $v_n$ is a viscosity solution to
        \begin{equation}\label{eq:unphisubsol1}
            \partial_t U(s,t)+H_\gamma(s,\partial_s U(s,t))=c,\qquad\text{on }(0,1)\times(-\delta,\delta),
        \end{equation}
        and standard stability properties of the viscosity solutions (see, e.g., \cite[Proposition II.2.2]{BardiCapuzzo-Dolcetta97}) show that also $v$ is a viscosity solution to~\cref{eq:unphisubsol1}. By definition we have that
        \begin{equation*}
            v(\ovs,0)=\unphi(\gamma(\ovs))\qquad\text{and}\qquad v(s,t)\le\unphi(\gamma(s)),\quad\text{for any }(s,t)\in[0,1]\times[-\delta,\delta],
        \end{equation*}
        therefore $\varphi$ is a supertangent to $v$ at $(\ovs,0)$. Since $v$ is a viscosity solution to~\cref{eq:unphisubsol1} it follows that
        \begin{equation*}
            H_\gamma(\ovs,\partial_s\varphi(\ovs))\le c,
        \end{equation*}
        which, since $\varphi$ and $\ovs$ are arbitrary, proves that $\unphi\circ\gamma$ is a subsolution to~\ghrefc{}.
    \end{proof}

    The next results concern the behavior of subsolution to~\hrefc{} and elements of $\omega_\Scal$ on static curves.

    \begin{lem}\label{omegaSconstaux}
        Given $\phi\in C(\Gamma)$ and a flux limiter satisfying~\cref{eq:fluxcrit}, let $f\in\omega_\Scal(\phi)$ and $w$ be a subsolution to~\hrefc{}. For any periodic static curve $\zeta$ the function
        \begin{equation*}
            t\longmapsto(\Scal(t)f)(\zeta(t))+ct-w(\zeta(t))
        \end{equation*}
        is constant.
    \end{lem}
    \begin{proof}
        Let $\{t_n\}_{n\in\Nds}$ and $\{t_n'\}_{n\in\Nds}$ be two positive diverging sequences such that $\lim_n\zeta(t_n)=\zeta(0)$ and $\lim_n\|\Scal(t_n')\phi+ct_n'-f\|_\infty=0$. We also assume, without loss of generality, that $\lim_n t_n'-t_n=\infty$ and $\Scal(t_n'-t_n)\phi+c(t_n'-t_n)$ uniformly converges to $g\in\omega_\Scal(\phi)$. It follows from~\cref{eq:vft} that
        \begin{align*}
            \left\|\Scal\left(t_n'\right)\phi+ct_n'-\Scal(t_n)g-ct_n\right\|_\infty=&\,\left\|\Scal\left(t_n+t_n'-t_n\right)\phi-\Scal(t_n)g+c\left(t_n'-t_n\right)\right\|_\infty\\
            \le&\,\left\|\Scal\left(t_n'-t_n\right)\phi+c\left(t_n'-t_n\right)-g\right\|_\infty,
        \end{align*}
        which shows that
        \begin{equation}\label{eq:omegaSconstaux1}
            \lim_{n\to\infty}\|\Scal(t_n)g+ct_n-f\|_\infty=0.
        \end{equation}
        Next we have by \cref{subsolextstatclass} that, for any $t_2\ge t_1\ge0$,
        \begin{align*}
            (\Scal(t_2)g)(\zeta(t_2))+ct_2-(\Scal(t_1)g)(\zeta(t_1))-ct_1\le&\,\int_{t_1}^{t_2}\left(L\left(\zeta,\dot\zeta\right)+c\right)d\tau=S_c(\zeta(t_1),\zeta(t_2))\\
            =&\,w(\zeta(t_2))-w(\zeta(t_1))
        \end{align*}
        and consequently that $t\mapsto(\Scal(t)g)(\zeta(t))+ct-w(\zeta(t))$ is nonincreasing. This monotonicity and~\cref{eq:Scal2bound} imply the existence of a $C\in\Rds$ such that
        \begin{equation}\label{eq:omegaSconstaux2}
            \lim_{t\to\infty}(\Scal(t)g)(\zeta(t))+ct-w(\zeta(t))=C.
        \end{equation}
        Finally we have by~\cref{eq:omegaSconstaux1,eq:omegaSconstaux2} that, for any $t\in\Rds^+$,
        \begin{equation*}
            C=\lim_{n\to\infty}(\Scal(t+t_n)g)(\zeta(t+t_n))+c(t+t_n)-w(\zeta(t+t_n))=(\Scal(t)f)(\zeta(t))+ct-w(\zeta(t)).
        \end{equation*}
    \end{proof}

    \begin{lem}\label{approxcurve}
        Let $\zeta$ be a static curve and define, for each $\rho\in(0,1)$, $\zeta_\rho(t)\coloneqq\zeta(\rho t)$. Then
        \begin{equation}\label{eq:approxcurve.1}
            \int_{t_1}^{t_2}\left(L\left(\zeta_\rho,\dot\zeta_\rho\right)+c\right)d\tau\le S_c(\zeta_\rho(t_1),\zeta_\rho(t_2))+o(1-\rho),\qquad\text{for any }t_2\ge t_1,
        \end{equation}
        where $o(\cdot)$ is the Landau symbol.
    \end{lem}
    \begin{proof}
        We preliminarily define the set $E$ made up by the $t\in\Rds$ such that $\zeta$ is differentiable in $t$, $\dot\zeta(t)\ne0$ and $\zeta(t)\notin\Vbf$. If $t\in E$ there is a $\gamma\in\Ecal$ and an $s\in(0,1)$ such that $\zeta(t)=\gamma(s)$ and
        \begin{equation}\label{eq:approxcurve1}
            L\left(\zeta(t),\dot\zeta(t)\right)+c=\sigma_c\left(\zeta(t),\dot\zeta(t)\right)=\sigma_{\gamma,c}^+(s)\frac{\dot\zeta(t)\cdot\dot\gamma(s)}{|\dot\gamma(s)|_2^2},
        \end{equation}
        therefore we have that $q\mapsto\sigma_c(\zeta(t),q)$ is differentiable in $\dot\zeta(t)$ and
        \begin{equation*}
            \dot\zeta(t)\cdot\partial_q\sigma_c\left(\zeta(t),\dot\zeta(t)\right)=\sigma_c\left(\zeta(t),\dot\zeta(t)\right),
        \end{equation*}
        In particular \cref{eq:approxcurve1,laglbound} yield that $q\mapsto\sigma_c(\zeta(t),q)$ is a subtangent to $q\mapsto L(\zeta(t),q)+c$ at $\dot\zeta(t)$ for all $t\in E$, thus, see \cite[Proposition 2.2.7]{Clarke90},
        \begin{equation}\label{eq:approxcurve2}
            \sigma_c\left(\zeta(t),\dot\zeta(t)\right)\in\dot\zeta(t)\cdot\partial_q L\left(\zeta(t),\dot\zeta(t)\right),\qquad\text{for any }t\in E.
        \end{equation}
        We stress out that $\partial_q L$ denotes Clarke's generalized gradient of $L$ in the second variable. We set, for any $\rho\in(0,1)$, the function $\ell_\rho:\Rds\to\Rds$ such that $\ell_\rho(t)$ is the projection of $\sigma_c\left(\zeta_\rho(t),\dot\zeta_\rho(t)\right)$ on $\dot\zeta_\rho(t)\cdot\partial_q L\left(\zeta_\rho(t),\dot\zeta_\rho(t)\right)$ whenever $\rho t\in E$ and
        \begin{equation*}
            \ell_\rho(t)\coloneqq\sigma_c\left(\zeta_\rho(t),\dot\zeta_\rho(t)\right),\qquad\text{otherwise}.
        \end{equation*}
        By \cite[Theorem 1.3.28]{Molchanov17} the functions $\ell_\rho$ are measurable, and from \cref{lagparreg} and \cite[Theorem 2.8.1]{ClarkeLedyaevSternWolenski98} we get
        \begin{equation}\label{eq:approxcurve3}
            \lim_{\rho\to1^-}\ell_\rho(t)=\sigma_c\left(\zeta(t),\dot\zeta(t)\right),\qquad\text{for a.e. }t\in\Rds.
        \end{equation}
        Thanks to \cite[Proposition 2.4.3]{ClarkeLedyaevSternWolenski98} we have that
        \begin{equation*}
            L\left(\zeta(\rho t),\rho\dot\zeta(\rho t)\right)-L\left(\zeta(\rho t),\dot\zeta(\rho t)\right)\le\ell_\rho(t)(\rho-1),\qquad\text{for any }\rho\in(0,1),\,\rho t\in E,
        \end{equation*}
        therefore it follows from~\cref{eq:approxcurve1} that, for any $\rho\in(0,1)$ and $\rho t\in E$,
        \begin{equation}\label{eq:approxcurve4}
            L\left(\zeta_\rho(t),\dot\zeta_\rho(t)\right)+c\le\sigma_c\left(\zeta(\rho t),\dot\zeta(\rho t)\right)+\ell_\rho(t)(\rho-1).
        \end{equation}
        Next we define $E_0$ as the set made up by the $t\in\Rds$ such that $\dot\zeta(t)=0$, then it is apparent that, for any $\rho\in(0,1)$ and a.e. $\rho t\in E_0$,
        \begin{equation}\label{eq:approxcurve5}
            L\left(\zeta_\rho(t),\dot\zeta_\rho(t)\right)+c=L\left(\zeta(\rho t),\dot\zeta(\rho t)\right)+c=\sigma_c\left(\zeta(\rho t),\dot\zeta(\rho t)\right).
        \end{equation}
        Notice that $\Rds\setminus(E\cup E_0)$ is a set of measure zero, thus~\cref{eq:approxcurve4}, \cref{eq:approxcurve5} and the positive homogeneity of $\sigma_c$ in the second variable yield, for any $\rho\in(0,1)$ and a.e. $t\in\Rds$,
        \begin{equation}\label{eq:approxcurve6}
            L\left(\zeta_\rho(t),\dot\zeta_\rho(t)\right)+c\le\sigma_c\left(\zeta_\rho(t),\dot\zeta_\rho(t)\right)+\left(\frac1\rho\sigma_c\left(\zeta_\rho(t),\dot\zeta_\rho(t)\right)-\ell_\rho(t)\right)(1-\rho).
        \end{equation}
        We point out that by \cref{lagparreg}, \cite[Corollary to Proposition 2.2.6]{Clarke90} and~\cref{eq:approxcurve2} there is a constant $M$ independent of $\rho$ such that
        \begin{equation*}
            \left|\sigma_c\left(\zeta(t),\dot\zeta(t)\right)\right|\le M\quad\text{and}\quad|\ell_\rho(t)|\le M\qquad\text{for a.e }t\in\Rds,
        \end{equation*}
        therefore~\cref{eq:approxcurve.1} follows from~\cref{eq:approxcurve3}, \cref{eq:approxcurve6} and the dominated convergence Theorem.
    \end{proof}

    Thanks to the previous \namecref{approxcurve} we can provide an extension of \cref{omegaSconstaux}.

    \begin{lem}\label{omegaSconst}
        Given $\phi\in C(\Gamma)$ and a flux limiter satisfying~\cref{eq:fluxcrit}, let $f\in\omega_\Scal(\phi)$ and $w$ be a subsolution to~\hrefc{}. For any periodic static curve $\zeta$ the function
        \begin{equation}\label{eq:omegaSconst.1}
            t\longmapsto f(\zeta(t))-w(\zeta(t))
        \end{equation}
        is constant.
    \end{lem}
    \begin{proof}
        We proceed by contradiction, assuming that~\cref{eq:omegaSconst.1} is not constant. We start noticing that $f$ and $w$ are both Lipschitz continuous by \cref{loclipconttsol,subsollip}, respectively, thus~\cref{eq:omegaSconst.1} is absolutely continuous since it is the combination of Lipschitz and absolutely continuous functions. It is then apparent that requiring~\cref{eq:omegaSconst.1} to be not constant is equivalent to ask that
        \begin{equation}\label{eq:omegaSconst1}
            |\{t\in\Rds:D(f(\zeta(t))-w(\zeta(t)))\ne0\}|>0.
        \end{equation}
        The periodicity of~\cref{eq:omegaSconst.1} is a trivial consequence of the periodicity of $\zeta$, which combined with~\cref{eq:omegaSconst1} implies that its derivative is negative at some points and positive at some others. We then assume without loss of generality that~\cref{eq:omegaSconst.1} is differentiable at 0 and that $m\coloneqq D(f(\zeta(0))-w(\zeta(0)))<0$. It follows that, for any $t$ in a neighborhood of 0,
        \begin{equation}\label{eq:omegaSconst2}
            f(\zeta(t))-w(\zeta(t))\le f(\zeta(0))-w(\zeta(0))+mt+o(t).
        \end{equation}
        Thanks to \cref{approxcurve} we have that, for any $t\ge0$ and $\rho\in(0,1)$,
        \begin{align*}
            (\Scal(t)f)(\zeta(t))+ct\le&\,f(\zeta((1-\rho)t))+\int_{\left(\frac1\rho-1\right)t}^{\frac t\rho}\left(L\left(\zeta_\rho,\dot\zeta_\rho\right)+c\right)d\tau\\
            \le&\,f(\zeta((1-\rho)t))+S_c(\zeta((1-\rho)t),\zeta(t))+o(1-\rho),
        \end{align*}
        hence \cref{subsolextstatclass} yields that for any $t\in\Rds^+$ and $\rho\in(0,1)$,
        \begin{equation}\label{eq:omegaSconst3}
            (\Scal(t)f)(\zeta(t))+ct-w(\zeta(t))\le f(\zeta((1-\rho)t))-w(\zeta((1-\rho)t))+o(1-\rho).
        \end{equation}
        Finally~\cref{eq:omegaSconst2,eq:omegaSconst3} show that fixed $t>0$ and for any $\rho$ sufficiently near 1
        \begin{equation*}
            (\Scal(t)f)(\zeta(t))+ct-w(\zeta(t))\le f(\zeta(0))-w(\zeta(0))+m(1-\rho)t+o(1-\rho),
        \end{equation*}
        therefore, since $m<0$, a suitable choice of $\rho$ proves that
        \begin{equation*}
            (\Scal(t)f)(\zeta(t))+ct-w(\zeta(t))<f(\zeta(0))-w(\zeta(0)),
        \end{equation*}
        in contradiction with \cref{omegaSconstaux}.
    \end{proof}

    \begin{lem}\label{approxuniflim}
        Let $\phi\in C(\Gamma)$, $c_x$ be a flux limiter satisfying~\cref{eq:fluxcrit} and $u$ be as in~\cref{eq:ltb.1}. If $f\in\omega_\Scal(\phi)$, $\zeta$ is a static curve and $\varepsilon>0$, then there is a $t\in\Rds^+$ such that
        \begin{equation*}
            |f(\zeta(t))-u(\zeta(t))|<\varepsilon.
        \end{equation*}
    \end{lem}
    \begin{proof}
        By definition there is a $z\in\Gamma$ such that
        \begin{equation*}
            u(\zeta(0))=\phi(z)+S_c(z,\zeta(0)),
        \end{equation*}
        then we choose an optimal curve $\xi:[0,T]\to\Gamma$ for $S_c(z,\zeta(0))$. Following \cref{approxlagpar} we also choose a curve $\xi_\varepsilon:[0,T_\varepsilon]\to\Gamma$ reparametrization of $\xi$ such that
        \begin{equation*}
            \int_0^T\sigma_c\left(\xi,\dot\xi\right)d\tau+\frac\varepsilon2\ge\int_0^{T_\varepsilon}\left(L\left(\xi_\varepsilon,\dot\xi_\varepsilon\right)+c\right)d\tau.
        \end{equation*}
        This implies that
        \begin{equation}\label{eq:approxuniflim1}
            u(\zeta(0))+\frac\varepsilon2\ge\phi(z)+\int_0^{T_\varepsilon}\left(L\left(\xi_\varepsilon,\dot\xi_\varepsilon\right)+c\right)d\tau\ge(\Scal(T_\varepsilon)\phi)(\zeta(0))+cT_\varepsilon.
        \end{equation}
        Next we let $\{t_n\}_{n\in\Nds}$ be a positive diverging sequence such that $\Scal(t_n)\phi+ct_n$ converges uniformly to $f$, then we have that for any $n$ big enough
        \begin{equation}\label{eq:approxuniflim2}
            \|\Scal(t_n)\phi+ct_n-f\|_\infty<\frac\varepsilon2\qquad\text{and}\qquad t_n>T_\varepsilon.
        \end{equation}
        We fix a $t_n$ satisfying~\cref{eq:approxuniflim2} and set $t\coloneqq t_n-T_\varepsilon$. We observe that by \cref{subsolextstatclass}
        \begin{equation*}
            u(\zeta(t))=u(\zeta(0))+S_c(\zeta(0),\zeta(t))=u(\zeta(0))+\int_0^t\left(L\left(\zeta,\dot\zeta\right)+c\right)d\tau.
        \end{equation*}
        This identity, \cref{Scalprop}, \cref{eq:approxuniflim1,eq:approxuniflim2}, yield
        \begin{align*}
            f(\zeta(t))-\frac\varepsilon2<&\,(\Scal(t_n)\phi)(\zeta(t))+ct_n=(\Scal(t)\Scal(T_\varepsilon)\phi)(\zeta(t))+c(t+T_\varepsilon)\\
            \le&\,(\Scal(T_\varepsilon)\phi)(\zeta(0))+cT_\varepsilon+\int_0^t\left(L\left(\zeta,\dot\zeta\right)+c\right)d\tau<u(\zeta(t))+\frac\varepsilon2
        \end{align*}
        which proves, together with~\cref{eq:omegalbound}, our claim.
    \end{proof}

    We can finally provide the proof of \cref{ltb}.

    \begin{proof}[Proof of \cref{ltb}]
        Let $\unphi$ be defined by~\cref{eq:limsuptsol}. \Cref{unphisubsol}, \cref{eq:omegalbound,eq:unphisup} show that $\unphi$ is a subsolution to~\hrefc{} satisfying
        \begin{equation}\label{eq:ltb1}
            \unphi\ge u,\qquad\text{on }\Gamma.
        \end{equation}
        Moreover, \cref{omegaSconst,approxuniflim} yield that $\unphi=u$ on the support all periodic static curves, thus, by \cref{subsolextstatclass,critcurveperiod}, $\unphi=u$ on $\wtAcal_\Gamma$. Finally~\cref{eq:ltb1} and the maximality of $u$, see \cref{eiksol}, prove that $\unphi=u$ on $\Gamma$. This concludes the proof thanks to~\cref{eq:omegalbound}.
    \end{proof}

    \subsection{The Supercritical Case}\label{supcritsec}

    This \lcnamecref{supcritsec} is devoted to the proof of \cref{supltb}. The argument is similar to the one used in \cref{convfintime} to prove \cref{ltbfinite}, therefore here we will just highlight the main steps.

    We start with the following \namecref{supweakltb}, whose proof is almost identical to the one given for \cref{weakltb}.

    \begin{prop}\label{supweakltb}
        Let $c_x$ be a flux limiter such that $\ova\coloneqq\max\limits_{x\in\Vbf}c_x>c$ and $u$ be a solution to~\cref{eq:globeik} in $\Gamma\setminus\Vbf_{\ova}$, where $\Vbf_{\ova}$ is defined as in~\cref{eq:supltb.1}, then $\Scal(t)u=u-\ova t$ on $\Gamma\times\Rds^+$.
    \end{prop}

    Proceeding as in the proof of \cref{subsolltb}, replacing \cref{liveoptcurve} with \cref{supliveoptcurve}, we get a convergence result when the initial datum is a subsolution:

    \begin{prop}\label{supsubsolltb}
        Given a flux limiter $c_x$ such that $\ova\coloneqq\max\limits_{x\in\Vbf}c_x>c$ and a subsolution $w$ to~\cref{eq:globeik}, we define
        \begin{equation*}
            u(x)\coloneqq\min_{y\in\Vbf_{\ova}}(w(y)+S_{\ova}(y,x)),\qquad\text{for }x\in\Gamma.
        \end{equation*}
        Then there is a time $T_w$, depending on $w$ and $\ova$, such that
        \begin{equation*}
            (\Scal(t)w)(x)+\ova t=u(x),\qquad\text{for any }x\in\Gamma,\,t\ge T_w.
        \end{equation*}
    \end{prop}

    Finally we have:

    \begin{proof}[Proof of \cref{supltb}]
        We point out that, since $\ova>c\ge a_0$, \cref{timeoflagpar} yields that every curve on $\Gamma$ has $\ova$--Lagrangian parametrization. Then, arguing as in the proof of \cref{ltbfinite} with straightforward modification, e.g., using \cref{supsubsolltb} instead of \cref{subsolltb}, we prove our claim.
    \end{proof}

    \section{Fixed Points of the Semigroup \texorpdfstring{$\Scal$}{S}}\label{fixpointsec}

    In this paper we have characterized the critical value $c$ dynamically, using closed curves on the network $\Gamma$. Alternatively $c$ can be seen as the minimum $a\in\Rds$ such that~\cref{eq:globeik} admits subsolutions. Both these characterizations are given in \cite{SiconolfiSorrentino18}. In more traditional settings, additional characterizations are known. In particular, see for instance \cite{Fathi08}, on compact connected Riemannian manifolds the critical value is the only value $a$ such that the semigroup $\phi\mapsto\Scal(t)\phi+at$ admits fixed points and these fixed points are the solutions to the respective Eikonal equation. In our case, however, the presence of the flux limiters influences this result.

    Indeed, if it is given a flux limiter $c_x$ satisfying~\cref{eq:fluxcrit}, \cref{ltb} shows that, for any $\phi\in C(\Gamma)$, $\Scal(t)\phi+ct$ converges to a continuous function as $t$ tends to $\infty$. This implies that $\Scal(t)\phi+bt$ diverges as $t$ tends to $\infty$ whenever $b\ne c$, i.e., $\phi\mapsto\Scal(t)\phi+bt$ does not admit fixed points. We know from \cref{weakltb} that the solutions to \hrefc{} in $\Gamma\setminus\wtVbf$ are fixed points of $\phi\mapsto\Scal(t)\phi+ct$, while the converse implication easily follows from the fact that, given a fixed point $\phi\in C(\Gamma)$, $\phi-ct$ is a solution to~\cref{eq:globteik}.\\
    Arguing in the same way we can further show that, if the flux limiter $c_x$ is such that $\ova\coloneqq\max\limits_{x\in\Vbf}c_x>c$, $\phi\mapsto\Scal(t)\phi+bt$ admits fixed points only if $b=\ova$ and these fixed points are exactly the solutions to~\cref{eq:globeik} in $\Gamma\setminus\Vbf_{\ova}$. More precisely we have the following \namecref{fixpoint}.

    \begin{theo}\label{fixpoint}\leavevmode
        \begin{myenum}
            \item Given a flux limiter $c_x$ satisfying~\cref{eq:fluxcrit}, the only value $b$ such that the semigroup $C(\Gamma)\ni\phi\mapsto\Scal(t)\phi+bt$ admits fixed points is the critical value $c$. These fixed points are the solutions to~\hrefc{} in $\Gamma\setminus\wtVbf$.
            \item Given a flux limiter $c_x$ such that $\ova\coloneqq\max\limits_{x\in\Vbf}c_x>c$, the only value $b$ such that the semigroup $C(\Gamma)\ni\phi\mapsto\Scal(t)\phi+bt$ admits fixed points is $\ova$. These fixed points are the solutions to~\cref{eq:globeik} in $\Gamma\setminus\Vbf_{\ova}$.
        \end{myenum}
    \end{theo}

    \appendix

    \section{Reparametrizations of Curves}\label{repcurvesec}

    Solutions to the time-dependent problem~\cref{eq:globteik} are given through a Lax--Oleinik type operator, while the solutions to the Eikonal problem~\hrefc{} are identified via a Hopf--Lax type formula exploiting the weak KAM theory. It is then clear that in order to perform our asymptotic analysis we need to establish a relationship between these two representation formulas. Following~\cite{DaviniSiconolfi06,FathiSiconolfi05}, this is done through reparametrizations of curves on $\Gamma$.\\
    In addition to their relevance for the asymptotic analysis, these results are also crucial for the proof of \cref{minactlip} and consequently of the local Lipschitz continuity of the solutions to the evolutive problem, see \cref{loclipconttsol}.

    \begin{defin}
        Given an absolutely continuous curve $\xi:[0,T]\to\Rds^N$, a curve $\zeta:[0,T']\to\Rds^N$ is called a \emph{reparametrization} of $\xi$ if there exists a nondecreasing surjective absolutely continuous function $\psi$ from $[0,T']$ onto $[0,T]$ with
        \begin{equation*}
            \zeta(t)=\xi\circ\psi(t),\qquad\text{for any }t\in\left[0,T'\right].
        \end{equation*}
    \end{defin}

    Note that if $\zeta$ is a reparametrization of $\xi$, the converse property in general is not true for $\psi$ could have not strictly positive derivative for a.e. $t$, see Zarecki criterion for an absolutely continuous inverse in~\cite{Bernal18}. We have that reparametrizations are absolutely continuous:

    \begin{lem}\label{reparac}
        \emph{\cite[Corollary 4]{SerrinVarberg69}} Let $\xi:[0,T]\to\Rds^N$ be a curve and $\psi:[0,T']\to[0,T]$ be absolutely continuous and nondecreasing. Then the reparametrization $\zeta\equiv\xi\circ\psi$ of $\xi$ is absolutely continuous and
        \begin{equation*}
            \frac d{dt}\zeta(t)=\dot\xi(\psi(t))\cdot\dot\psi(t),\qquad\text{a.e.\ in }\left[0,T'\right].
        \end{equation*}
    \end{lem}

    \begin{lem}\label{repsigma}
        If the curve $\zeta:[0,T']\to\Gamma$ is a reparametrization of a curve $\xi:[0,T]\to\Gamma$, then
        \begin{equation*}
            \int_0^{T'}\sigma_a\left(\zeta,\dot\zeta\right)d\tau=\int_0^T\sigma_a\left(\xi,\dot\xi\right)d\tau,\qquad\text{for every }a\in\Rds.
        \end{equation*}
    \end{lem}
    \begin{proof}
        It follows from the definition that $(x,q)\mapsto\sigma_a(x,q)$ is positively homogeneous on $q$, thus, if we let $\psi$ be the nondecreasing absolutely continuous function such that $\zeta\equiv\xi\circ\psi$ and consider the change of variable $r=\psi(\tau)$, we get from \cref{reparac} that, for every $a\in\Rds$,
        \begin{equation*}
            \int_0^{T'}\sigma_a\left(\zeta,\dot\zeta\right)d\tau=\int_0^{T'}\sigma_a\left(\xi\circ\psi,\dot\xi\circ\psi\right)\cdot\dot\psi(\tau)d\tau=\int_0^T\sigma_a\left(\xi,\dot\xi\right)dr.
        \end{equation*}
    \end{proof}

    The next \namecref{repconstspeed} comes from classical results of analysis in metric space, see~\cite{Bernal18} and \cite[Lemma 3.11]{Davini07}.

    \begin{prop}\label{repconstspeed}
        Any curve in $[0,T]$ is the reparametrization of a curve $\xi$ with constant speed, namely with $|\dot\xi|_2\equiv\mathrm{constant}$ a.e., defined on a bounded interval.
    \end{prop}

    \begin{defin}\label{lagpardef}
        Given a curve $\xi:[0,T]\to\Gamma$ and an $a\in\Rds$ we say that $\xi$ has an \emph{$a$--Lagrangian parametrization} if
        \begin{equation*}
            L\left(\xi(t),\dot\xi(t)\right)+a=\sigma_a\left(\xi(t),\dot\xi(t)\right),\qquad\text{for a.e. }t\in[0,T].
        \end{equation*}
        We will also say that $\zeta$ is an $a$--Lagrangian reparametrization of $\xi$ if $\zeta$ has an $a$--Lagrangian parametrization and it is a reparametrization of $\xi$.
    \end{defin}

    \begin{prop}\label{lagparreg}
        If $\xi$ has an $a$--Lagrangian parametrization there is a minimal constant $\kappa_a$, depending only on $a$, such that $\xi$ is $\kappa_a$--Lipschitz continuous. Furthermore, if $a<b$, then $\kappa_a<\kappa_b$.
    \end{prop}
    \begin{proof}
        We start assuming that there exist an arc $\gamma$ and a curve $\eta:[0,T]\to[0,1]$ such that $\xi=\gamma\circ\eta$. We have that for a.e. $t\in[0,T]$
        \begin{equation}\label{eq:lagparreg1}
            L_\gamma(\eta(t),\dot\eta(t))=\mu(t)\dot\eta(t)-H_\gamma(\eta(t),\mu(t))\quad\text{or}\quad\dot\eta(t)=0,
        \end{equation}
        where $\mu(t)$ satisfies $H_\gamma(\eta(t),\mu(t))=a$. It follows that, for all $t$ satisfying~\cref{eq:lagparreg1}, $\dot\eta(t)\in\partial_\mu H_\gamma(\eta(t),\mu(t))$ and, by the coercivity of $H_\gamma$ in $\mu$, $|\mu(t)|\le M$ for some $M>0$. Since $H_\gamma(s,\mu)$ is locally Lipschitz continuous in $\mu$ uniformly with respect to $s$ and $\gamma$, see \cite[Corollary to Proposition 2.2.6]{Clarke90}, we find a constant $C_M$ with $|\dot\eta|\le C_M$ a.e.. This yields the existence of a minimal constant $\kappa_a$, depending only on $a$, such that $|\dot\xi|_2\le\kappa_a$ a.e.. Moreover, if $a<b$, \labelcref{condqconv} implies that $\kappa_a<\kappa_b$. Finally our claim is a consequence of \cref{curveascomp}.
    \end{proof}

    The next \namecref{laglbound}, together with \cref{repsigma}, shows that, given an upper bound for the flux limiter, Lagrangian reparametrizations are, in a certain sense, optimal among all the possible reparametrizations.

    \begin{lem}\label{laglbound}
        Assume that $c_x\le a$ for all $x\in\Vbf$, then
        \begin{equation}\label{eq:laglbound.1}
            L(x,q)+a\ge\sigma_a(x,q),\qquad\text{for any }(x,q)\in T\Gamma.
        \end{equation}
    \end{lem}
    \begin{proof}
        If $x\in\Vbf$ we have from our assumptions that
        \begin{equation}\label{eq:laglbound1}
            L(x,0)=-c_x\ge-a=\sigma_a(x,0)-a.
        \end{equation}
        Next we let $(x,q)\in T\Gamma$ with $q\ne0$, it then follows that there is an arc $\gamma$ such that, putting for notation's sake $s\coloneqq\gamma^{-1}(x)$ and $\lambda\coloneqq\dfrac{\dot\gamma(s)\cdot q}{|\dot\gamma(s)|_2^2}$,
        \begin{equation}\label{eq:laglbound2}
            L(x,q)=L_\gamma(s,\lambda)\ge\max\{\mu\lambda-a:\mu\in\Rds,\,H_\gamma(s,\mu)=a\}\ge\sigma_a(x,q)-a.
        \end{equation}
        Finally~\cref{eq:laglbound1,eq:laglbound2} yields~\cref{eq:laglbound.1}.
    \end{proof}

    \begin{defin}\label{admisdef}
        Given a curve $\xi:[0,T]\to\Gamma$, we set
        \begin{equation*}
            a_\xi\coloneqq-\min_{t\in[0,T]}L(\xi(t),0),
        \end{equation*}
        then we say that $a$ is \emph{admissible} for $\xi$ if $a>a_\xi$. Trivially, if $a>a_0$, it is admissible for any curve on $\Gamma$.
    \end{defin}

    The concept of admissibility is strongly related to Lagrangian reparametrizations, as shown by the next \namecref{timeoflagpar}.

    \begin{theo}\label{timeoflagpar}
        Let $\xi:[0,T]\to\Gamma$ be a curve with a.e.\ non-vanishing derivative, $a_\xi$ be as in \cref{admisdef} and define for each $a\ge a_\xi$ the set
        \begin{equation*}
            T(a)\coloneqq\left\{t>0:\text{$\xi$ has an $a$--Lagrangian reparametrization $\zeta:[0,t]\to\Gamma$}\right\}.
        \end{equation*}
        The following facts hold:
        \begin{myenum}
            \item if $a$ is admissible for $\xi$ then $T(a)$ is a compact interval, namely,
            \begin{equation*}
                T(a)=\left[\underline T(a),\ovT(a)\right],\qquad\text{for some }\ovT(a)\ge\underline T(a)>0;
            \end{equation*}
            \item if $a$ and $b$ are both admissible for $\xi$ and $b>a$, then $\ovT(b)\le\underline T(a)$;
            \item $\lim\limits_{a\to\infty}\ovT(a)=0$ and, for any admissible $a$,
            \begin{equation*}
                \underline T(a)=\lim_{b\to a^+}\ovT(b),\qquad\ovT(a)=\lim_{b\to a^-}\underline T(b);
            \end{equation*}
            \item if $\underline T(a_\xi)\coloneqq\lim\limits_{a\to a_\xi^+}\ovT(a)$ is finite, then
            \begin{equation*}
                T(a_\xi)=[\underline T(a_\xi),\infty).
            \end{equation*}
        \end{myenum}
        In particular, for any $t\in(0,\infty)$, there exists an $a\ge a_\xi$ such that $\xi$ has an $a$--Lagrangian reparametrization $\zeta:[0,t]\to\Gamma$.
    \end{theo}
    \begin{proof}
        If there exist an arc $\gamma$ and a curve $\eta:[0,T]\to[0,1]$ such that $\xi=\gamma\circ\eta$, then our claim is a consequence of \cite[Proposition 3.13 and Remark 3.17]{Davini07} applied to the curve $\eta$ and the Lagrangian $L_\gamma$. It is also shown there that for any admissible $a$ there is a $C_a>0$, independent of $\eta$, such that $\ovT(a)\le C_a T$. Since the arcs are finite $C_a$ can be chosen to be independent of the arc $\gamma$.\\
        In the general case we have by \cref{curveascomp} that there is an at most countable collection of open disjoint intervals $\{I_i\}$ with $\bigcup_i\ovI_i=[0,T]$ such that
        \begin{equation*}
            \xi\left(\ovI_i\right)\subseteq\gamma_i([0,1]),\qquad\text{for each index }i,
        \end{equation*}
        where $\gamma_i$ is an arc of the network. Setting $\eta_i\coloneqq\gamma^{-1}_i\circ\xi|_{\ovI_i}$ for every index $i$ we can assume that up to a translation $\eta_i$ is defined on an interval $[0,T_i]$. Since $\xi$ has non-vanishing derivative we get
        \begin{equation*}
            \int_0^T L\left(\xi,\dot\xi\right)d\tau=\sum_i\int_0^{T_i}L_{\gamma_i}(\eta_i,\dot\eta_i)d\tau.
        \end{equation*}
        If we define $\xi_i\coloneqq\gamma_i\circ\eta_i$ for each index $i$, we have by the previous step that our claim is true for each $\xi_i$. Moreover, setting
        \begin{equation*}
            T_i(a)\coloneqq\left\{t>0:\text{$\xi_i$ has an $a$--Lagrangian reparametrization $\zeta:[0,t]\to\Gamma$}\right\},
        \end{equation*}
        we have that whenever $a$ is admissible for $\xi$ it is admissible for each $\xi_i$ and
        \begin{equation*}
            T_i(a)=\left[\underline T_i(a),\ovT_i(a)\right].
        \end{equation*}
        Then our claim will follows if $\sum_i\ovT_i(a)$ is finite for any admissible $a$. To prove this we observe that each $\xi$ is defined on an interval $[0,T_i]$ with $\sum_i T_i=T$ and that by the previous step there is, for any admissible $a$, a constant $C_a$ such that $\ovT_i(a)\le C_a T_i$, which implies that $\sum_i\ovT_i(a)\le C_a T$.
    \end{proof}

    \begin{cor}\label{approxlagpar}
        Let $\xi:[0,T]\to\Gamma$ be a curve with a.e.\ non-vanishing derivative and define for each $t>0$
        \begin{equation*}
            [\xi]_t\coloneqq\{\zeta:[0,t]\to\Gamma:\text{$\zeta$ is a reparametrization of $\xi$}\}.
        \end{equation*}
        Then, if~\cref{eq:fluxcrit} holds,
        \begin{equation*}
            \int_0^T\sigma_c\left(\xi,\dot\xi\right)d\tau=\inf\left\{\int_0^t\left(L\left(\zeta,\dot\zeta\right)+c\right)d\tau:\zeta\in[\xi]_t,t>0\right\}.
        \end{equation*}
    \end{cor}
    \begin{proof}
        We let $\{a_n\}_{n\in\Nds}$ be a decreasing sequence converging to $c$, then
        \begin{equation}\label{eq:approxlagpar1}
            \lim_{n\to\infty}\int_0^T\sigma_{a_n}\left(\xi,\dot\xi\right)d\tau=\int_0^T\sigma_c\left(\xi,\dot\xi\right)d\tau
        \end{equation}
        by \cref{sigcont} and the monotone convergence Theorem. Each $a_n$ is bigger than $c$, thus~\cref{eq:loclagmin,eq:fluxcrit} yield that it is admissible for $\xi$. \Cref{timeoflagpar} then implies that, for each $n\in\Nds$, there is an $a_n$--Lagrangian reparametrization $\zeta_n:[0,T_n]\to\Gamma$ of $\xi$. It follows from \cref{repsigma} that
        \begin{equation*}
            \int_0^T\sigma_{a_n}\left(\xi,\dot\xi\right)d\tau=\!\int_0^{T_n}\sigma_{a_n}\left(\zeta_n,\dot\zeta_n\right)d\tau=\!\int_0^{T_n}\!\left(L\left(\zeta_n,\dot\zeta_n\right)+a_n\right)d\tau\ge\!\int_0^{T_n}\left(L\left(\zeta_n,\dot\zeta_n\right)+c\right)d\tau,
        \end{equation*}
        we can then conclude thanks to~\cref{eq:approxlagpar1,laglbound}.
    \end{proof}

    \section{Lipschitz Continuity of the Minimal Action}\label{minact}

    We consider the minimal action
    \begin{equation}\label{eq:minact}
        h_T(y,x)=\min\left\{\int_0^T L\left(\xi,\dot\xi\right)d\tau:\text{$\xi$ is a curve with $\xi(0)=y$, $\xi(T)=x$}\right\},
    \end{equation}
    for $(y,x,T)\in\Gamma^2\times\Rds^+$. In this \lcnamecref{minact} we will provide a Lipschitz continuity result for the minimal action using Lagrangian parametrizations. This is an improvement with respect to \cite[Theorem 5.4]{PozzaSiconolfi23}, which proves the continuity of $(y,x,T)\mapsto h_T(y,x)$.

    \begin{lem}\label{lagparisopt}
        Given $(y,x,T)\in\Gamma^2\times\Rds^+$ there exist an optimal curve $\zeta$ for $h_T(y,x)$ and a constant $a\ge a_\zeta$ such that $\zeta$ has an $a$--Lagrangian parametrization.
    \end{lem}
    \begin{proof}
        Given $(y,x,T)\in A_C$ there is an optimal curve $\xi:[0,T]\to\Gamma$ for $h_T(y,x)$. If $|\dot\xi|_2=0$ a.e., i.e., $y=x$, then, setting $a\coloneqq-L(x,0)$, $\xi$ has an $a$--Lagrangian parametrization. Otherwise, by \cref{repconstspeed}, it is a reparametrization of a curve $\zeta_0:[0,T']\to\Gamma$ with $|\dot\zeta_0|_2$ constant a.e.. Since by our assumptions $\zeta_0$ has a.e.\ non-vanishing derivative, \cref{timeoflagpar} yields the existence of a constant $a$ and an $a$--Lagrangian reparametrization $\zeta:[0,T]\to\Gamma$ of $\zeta_0$. In particular we have, thanks to \cref{repsigma,laglbound},
        \begin{equation*}
            \int_0^T\left(L\left(\xi,\dot\xi\right)+a\right)d\tau\ge\int_0^T\sigma_a\left(\xi,\dot\xi\right)d\tau=\int_0^T\sigma_a\left(\zeta,\dot\zeta\right)d\tau=\int_0^T\left(L\left(\zeta,\dot\zeta\right)+a\right)d\tau,
        \end{equation*}
        which shows that $\zeta$ is optimal.
    \end{proof}

    \begin{lem}\label{lipcur}
        Given $C>0$ and
        \begin{equation}\label{eq:lipcur.1}
            A_C\coloneqq\left\{(y,x,T)\in\Gamma^2\times\Rds^+:d_\Gamma(y,x)\le CT\right\},
        \end{equation}
        we have that for all $(y,x,T)\in A_C$ there is a constant $\kappa$, depending only on $C$, such that there exists an optimal curve for $h_T(y,x)$ which is Lipschitz continuous of rank $\kappa$.
    \end{lem}
    \begin{proof}
        We notice that, if $(y,x,T)\in A_C$, there is a curve $\xi:[0,T]\to\Gamma$ with $|\dot\xi|_2\le C$ such that $\xi(0)=y$ and $\xi(T)=x$. Consequently, setting $M_1\coloneqq\sup\limits_{x\in\Gamma,|q|_2\le C}L(x,q)$, we get
        \begin{equation}\label{eq:lipcur1}
            \int_0^T L\left(\xi,\dot\xi\right)d\tau\le T\sup_{x\in\Gamma,|q|_2\le C}L(x,q)=M_1T.
        \end{equation}
        Since $L$ is a superlinearly coercive function we can choose two positive $A$ and $B$ such that
        \begin{equation*}
            A|q|_2-B\le L(x,q),\qquad\text{for any }(x,q)\in T\Gamma,
        \end{equation*}
        thus, if $\zeta:[0,T]\to\Gamma$ is an optimal curve for $h_T(y,x)$, we have by~\cref{eq:lipcur1}
        \begin{equation*}
            A\int_0^T\left|\dot\zeta(\tau)\right|_2d\tau-BT\le\int_0^T L\left(\zeta,\dot\zeta\right)d\tau\le\int_0^T L\left(\xi,\dot\xi\right)d\tau\le M_1T.
        \end{equation*}
        Setting $M_2\coloneqq\dfrac{M_1+B}A$, we then have that, whenever $(y,x,T)\in A_C$ and $\zeta$ is an optimal curve for $h_T(y,x)$,
        \begin{equation*}
            \int_0^T\left|\dot\zeta(\tau)\right|_2d\tau\le M_2T
        \end{equation*}
        and consequently
        \begin{equation}\label{eq:lipcur2}
            \left|\left\{t\in[0,T]:\left|\dot\zeta(t)\right|_2\le2M_2\right\}\right|\ge\frac T2.
        \end{equation}
        Given $(y,x,T)\in A_C$ we fix, thanks to \cref{lagparisopt}, an optimal curve $\zeta$ for $h_T(y,x)$ and a constant $a\ge a_\zeta$ such that $\zeta$ has an $a$--Lagrangian parametrization. $\zeta$ is differentiable a.e., therefore~\cref{eq:lipcur2} yields the existence of a $t'\in[0,T]$, a $\gamma'\in\Ecal$ and an $s'\in(0,1)$ such that $\zeta$ is differentiable in $t'$, $\left|\dot\zeta(t')\right|_2\le2M_2$, $\zeta(t')=\gamma'(s')$ and
        \begin{equation}\label{eq:lipcur3}
            L\left(\zeta\left(t'\right),\dot\zeta\left(t'\right)\right)+a=\sigma_a\left(\zeta\left(t'\right),\dot\zeta\left(t'\right)\right)=\sigma_{\gamma',a}^+\left(s'\right)\frac{\dot\zeta(t')\cdot\dot\gamma'(s')}{|\dot\gamma'(s')|_2^2}.
        \end{equation}
        Assuming that $\dot\zeta(t')\ne0$ we have that $q\mapsto\sigma_a(\zeta(t'),q)$ is differentiable in $\dot\zeta(t')$ and
        \begin{equation*}
            \dot\gamma'\left(s'\right)\cdot\partial_q\sigma_a\left(\zeta\left(t'\right),\dot\zeta\left(t'\right)\right)=\sigma_{\gamma',a}^+\left(s'\right).
        \end{equation*}
        Moreover, we have by~\cref{eq:lipcur3,laglbound} that $q\mapsto\sigma_a(\zeta(t'),q)$ is a subtangent to $q\mapsto L(\zeta(t'),q)+a$ at $\dot\zeta(t')$, therefore, see for instance \cite[Proposition 2.2.7]{Clarke90},
        \begin{equation*}
            \sigma_{\gamma',a}^+(s)\in\dot\gamma'\left(s'\right)\cdot\partial_q L\left(\zeta\left(t'\right),\dot\zeta\left(t'\right)\right).
        \end{equation*}
        If instead $\dot\zeta(t')=0$ then
        \begin{equation*}
            a=-L_{\gamma'}\left(s',0\right)=\min_{\mu\in\Rds}H_{\gamma'}\left(s',\mu\right)\le\max_{s\in[0,1]}\min_{\mu\in\Rds}H_{\gamma'}(s,\mu)=a_{\gamma'},
        \end{equation*}
        thus $\sigma_{\gamma',a}^+(s')\le\sigma_{\gamma',a_{\gamma'}}^+(s')$. In both cases, since $\left|\dot\zeta(t')\right|_2\le2M_2$, we further have by \cite[Corollary to Proposition 2.2.6]{Clarke90} that
        \begin{equation*}
            \left|\sigma_{\gamma',a}^+\left(s'\right)\right|\le\sup_{\gamma\in\Ecal,s\in[0,1],|q|_2\le2M_2+2}(|\dot\gamma(s)|_2|L(\gamma(s),q)|)\vee\left|\sigma_{\gamma,a_\gamma}^+(s)\right|\eqqcolon M_3,
        \end{equation*}
        which in turn implies that
        \begin{equation*}
            a=H_\gamma\left(s',\sigma_{\gamma',a}^+\left(s'\right)\right)\le\max_{\gamma\in\Ecal,s\in[0,1],|\mu|\le M_3}H_\gamma(s,\mu)\eqqcolon a^*.
        \end{equation*}
        It follows from \cref{lagparreg} that there is a constant $\kappa$ depending on $a^*$ such that $|\dot\zeta|_2\le\kappa$ a.e.. We conclude this proof observing that the constant $a^*$ only depends on $C$, thus, for each $(y,x,T)\in A_C$, it is always possible to select an optimal curve for $h_T(y,x)$ which is also $\kappa$--Lipschitz continuous.
    \end{proof}

    \Cref{lipcur} is crucial for the proof of the next \namecref{minactlip}.

    \begin{prop}\label{minactlip}
        Let $A_C$ be defined by~\cref{eq:lipcur.1}, then there is a constant $\ell$ such that the minimal action in~\cref{eq:minact} is Lipschitz continuous of rank $\ell$ on $A_C$.
    \end{prop}
    \begin{proof}
        We fix $(y,x,T)\in A_C$, then \cref{lipcur} shows that there is a constant $\kappa$, depending only on $C$, and an optimal curve $\xi$ for $h_T(y,x)$ such that $\xi$ is $\kappa$--Lipschitz continuous. We set
        \begin{equation*}
            \ell'\coloneqq\sup_{(x,q)\in T\Gamma,|q|_2\le2\kappa}L(x,q),\qquad l\coloneqq\inf_{(x,q)\in T\Gamma}L(x,q)\qquad\text{and}\qquad\ell''\coloneqq3\ell'-2l.
        \end{equation*}
        We start proving that, if $(y,x,T')\in A_C$,
        \begin{equation}\label{eq:minactlip1}
            |h_T(y,x)-h_{T'}(y,x)|\le\ell''\left|T-T'\right|.
        \end{equation}
        We temporarily assume that
        \begin{equation}\label{eq:minactlip2}
            \left|T'-T\right|<\frac{T\wedge T'}2,
        \end{equation}
        then we define
        \begin{equation*}
            \ovxi(t)\coloneqq\left\{
            \begin{aligned}
                &\xi(t),&&\text{if }t\in\left[0,T-2\left|T-T'\right|\right),\\
                &\xi\left(2\left|T-T'\right|\frac{t-T+2|T-T'|}{T'-T+2|T-T'|}+T-2\left|T-T'\right|\right),&&\text{if }t\in\left[T-2\left|T-T'\right|,T'\right],
            \end{aligned}
            \right.
        \end{equation*}
        which is $2\kappa$--Lipschitz continuous curve connecting $y$ to $x$, thus
        \begin{align*}
            h_{T'}(y,x)-h_T(y,x)\le\,&\int_0^{T'}L\left(\ovxi,\dot\ovxi\right)d\tau-\int_0^T L\left(\xi,\dot\xi\right)d\tau\\
            =\,&\int_{T-2|T'-T|}^{T'}L\left(\ovxi,\dot\ovxi\right)d\tau-\int_{T-2|T'-T|}^T L\left(\xi,\dot\xi\right)d\tau\\
            \le\,&3\ell'\left|T-T'\right|-2l\left|T-T'\right|\le\ell''\left|T-T'\right|.
        \end{align*}
        Interchanging the roles of $T$ and $T'$ we get that~\cref{eq:minactlip1} holds true whenever~\cref{eq:minactlip2} is satisfied.\\
        Now assume that~\cref{eq:minactlip2} does not hold, then we pretend without loss of generality that $T'>T$ and choose an integer $m$ such that
        \begin{equation*}
            \frac{|T'-T|}m<\frac{T\wedge T'}2=\frac T2.
        \end{equation*}
        We define the sequence $\{T_i\}_{i=0}^m$ such that
        \begin{equation*}
            T_0\coloneqq T,\qquad T_i\coloneqq T_{i-1}+\frac{|T'-T|}m
        \end{equation*}
        and observe that $T_m=T'$. By definition $T_{i-1}$ and $T_i$ satisfy~\cref{eq:minactlip2}, consequently we have from the previous step that
        \begin{equation*}
            |h_T(y,x)-h_{T'}(y,x)|\le\sum_{i=1}^m\left|h_{T_i}(y,x)-h_{T_{i-1}}(y,x)\right|\le\sum_{i=1}^m\ell''|T_i-T_{i-1}|=\ell''\left|T-T'\right|.
        \end{equation*}
        To prove the general case we let $(y',x',T')$ be an element of $A_C$, then we define
        \begin{equation*}
            T_y\coloneqq\frac{d_\Gamma(y,y')}C\quad\text{and}\quad T_x\coloneqq\frac{d_\Gamma(x,x')}C.
        \end{equation*}
        We point out that by definition $(y,y',T_y),(x,x',T_x)\in A_C$, thus $(y',x',T+T_y+T_x)\in A_C$. In particular we have by \cref{lipcur} that there exist two Lipschitz continuous curves $\xi_x$, $\xi_y$ of rank $\kappa$ connecting, respectively, $x$ to $x'$ and $y'$ to $y$. We further define
        \begin{equation*}
            \ovxi(t)\coloneqq\left\{
            \begin{aligned}
                &\xi_y(t),&&\text{if }t\in[0,T_y),\\
                &\xi(t-T_y),&&\text{if }t\in[T_y,T+T_y),\\
                &\xi_x\left(t-T-T_y\right),&&\text{if }t\in[T+T_y,T+T_y+T_x].
            \end{aligned}
            \right.
        \end{equation*}
        Clearly $\ovxi:[0,T+T_y+T_x]\to\Gamma$ is a $\kappa$--Lipschitz continuous curve connecting $y'$ to $x'$, therefore, by~\cref{eq:minactlip1},
        \begin{align*}
            h_{T'}\left(y',x'\right)-h_T(y,x)\le&\,\left|h_{T'}\left(y',x'\right)-h_{T+T_y+T_x}\left(y',x'\right)\right|+h_{T+T_y+T_x}\left(y',x'\right)-h_T(y,x)\\
            \le&\,\ell''\left(\left|T-T'\right|+T_x+T_y\right)+\int_0^{T_y}L\left(\ovxi,\dot\ovxi\right)d\tau+\int_{T+T_y}^{T+T_y+T_x}L\left(\ovxi,\dot\ovxi\right)d\tau\\
            \le&\,\ell''\left(\left|T-T'\right|+T_x+T_y\right)+\ell'(T_x+T_y)\\
            \le&\,\ell''\left(\left|T-T'\right|+\frac2C\left(d_\Gamma\left(y,y'\right)+d_\Gamma\left(x,x'\right)\right)\right).
        \end{align*}
        Finally, interchanging the roles of $(y,x,T)$ and $(y',x',T')$, we prove our claim.
    \end{proof}

    \section{Proof of \crtcrefnamebylabel{liveoptcurve}~\ref{liveoptcurve}}\label{curvecostsec}

    In order to prove \cref{liveoptcurve} we need an auxiliary result.

    \begin{lem}\label{sigest}
        Let $\xi:[0,T]\to\Gamma$ be a curve such that
        \begin{equation}\label{eq:sigest.1}
            \xi([0,T])\cap(\Acal_\Gamma\setminus\Vbf)=\emptyset.
        \end{equation}
        Then there exist two positive constants $A$ and $B$ independent of $\xi$ such that
        \begin{equation}\label{eq:sigest.2}
            A\int_0^T\left|\dot\xi(\tau)\right|_2d\tau-B\le\int_0^T\sigma_c\left(\xi,\dot\xi\right)d\tau.
        \end{equation}
    \end{lem}

    The proof of this \namecref{liveoptcurve} is divided in three steps. The first two are particular cases of \cref{sigest}.

    \begin{lem}\label{sigestaux}
        Let $\xi:[0,T]\to\Gamma$ be a curve such that
        \begin{equation}\label{eq:sigestaux.1}
            \xi=(\gamma_1\circ\eta_1)*\dotsb*(\gamma_k\circ\eta_k),
        \end{equation}
        where $\gamma_i\in\Ecal$, $\gamma_i((0,1))\cap\Acal_\Gamma=\emptyset$ and $\dot\eta_i=1$ a.e.\ for any $i\in\{1,\dotsc,k\}$. Then there exist two positive constants $A$ and $B$ independent of $\xi$ such that
        \begin{equation}\label{eq:sigestaux.2}
            AT-B\le\int_0^T\sigma_c\left(\xi,\dot\xi\right)d\tau.
        \end{equation}
    \end{lem}
    \begin{proof}
        We preliminarily define the set $\Ecal'\subset\Ecal$ made up by the arcs $\gamma$ such that $\gamma((0,1))\cap\Acal_\Gamma=\emptyset$ and
        \begin{equation}\label{eq:sigestaux1}
            B_1\coloneqq-\min_{s_1,s_2\in[0,1],\gamma\in\Ecal'}\int_{s_1}^{s_2}\sigma_{\gamma,c}^+(s)ds\ge0.
        \end{equation}
        Next we observe that, since the network is finite, there is only a finite number of closed curves
        \begin{equation}\label{eq:sigestaux2}
            \zeta_i\coloneqq\gamma_{i_1}*\dotsc*\gamma_{i_{k_i}}
        \end{equation}
        such that $\zeta_i$ is simple or $\zeta_i=\gamma*\widetilde\gamma$ and, for all $l\in\{1,\dotsc,k_i\}$, $\gamma_{i_l}\in\Ecal'$. Then we define
        \begin{equation}\label{eq:sigestaux3}
            A_1\coloneqq\min_i\int_0^{k_i}\sigma_c\left(\zeta_i,\dot\zeta_i\right)d\tau.
        \end{equation}
        We stress out that $A_1>0$ since the supports of the $\zeta_i$ are not contained in the Aubry set.\\
        If $\xi$ and the $\gamma_i$ in the statement are such that
        \begin{equation}\label{eq:sigestaux4}
            \xi=\gamma_1*\dotsb*\gamma_k
        \end{equation}
        we define $j,l\in\{1,\dotsc,k\}$ as the smallest indices such that $j<l$ and $\gamma_j(0)=\gamma_l(1)$. We assume that such $j$, $l$ exist and, to ease notation, we also assume that $j>1$ and $l<k$; the other cases can be treated with straightforward modifications. We set
        \begin{equation*}
            \zeta'_1\coloneqq\gamma_j*\dotsb*\gamma_l,\qquad\xi'_1\coloneqq\gamma_1*\dotsb*\gamma_{j-1},\qquad\xi'_2\coloneqq\gamma_{l+1}*\dotsb*\gamma_k,
        \end{equation*}
        then $\zeta'_1$ is as in~\cref{eq:sigestaux2}, $\xi'_1$ is simple and $\xi=\xi'_1*\zeta_1'*\xi_2'$. Iterating the above procedure a finite number of times we get that the support of $\xi$ is made up by the closed curves $\{\zeta_i'\}_{i=1}^m$ as in~\cref{eq:sigestaux2} and the non-closed simple curve
        \begin{equation}\label{eq:sigestaux5}
            \overline\xi\coloneqq\gamma_1'*\dotsb*\gamma_n',
        \end{equation}
        where $\{\gamma_i'\}_{i=1}^n$ is a subset of $\{\gamma_i\}_{i=1}^k\subseteq\Ecal'$. Therefore
        \begin{align*}
            \int_0^T\sigma_c\left(\xi,\dot\xi\right)d\tau=\,&\sum_{i=1}^m\int\sigma_c\left(\zeta'_i,\dot\zeta'_i\right)d\tau+\sum_{j=1}^n\int_0^1\sigma_c\left(\gamma'_j,\dot\gamma'_j\right)d\tau\\
            =\,&\sum_{i=1}^m\int\sigma_c\left(\zeta'_i,\dot\zeta'_i\right)d\tau+\sum_{j=1}^n\int_0^1\sigma_{c,\gamma_j'}^+(s)ds.
        \end{align*}
        It follows from~\cref{eq:sigestaux1,eq:sigestaux3} that, for any $i\in\{1,\dotsc,m\}$ and $j\in\{1,\dotsc,n\}$,
        \begin{equation*}
            \int\sigma_c\left(\zeta'_i,\dot\zeta'_i\right)d\tau\ge A_1\qquad\text{and}\qquad\int_0^1\sigma_{c,\gamma_j'}^+(s)ds\ge-B_1,
        \end{equation*}
        thus
        \begin{equation}\label{eq:sigestaux6}
            \int_0^T\sigma_c\left(\xi,\dot\xi\right)d\tau\ge mA_1-nB_1.
        \end{equation}
        It is left to provide an estimate for $m$ and $n$. We start setting $B_2$ as the number of elements in $\Ecal'$, then we infer from~\cref{eq:sigestaux5} that $B_2$ bounds from above $n$. Next we observe that $\xi$ contains at least one closed curve as in~\cref{eq:sigestaux2} if $k\ge B_2$. Moreover, if
        \begin{equation*}
            k\ge m'B_2,
        \end{equation*}
        $\xi$ contains at least $m'$ closed curve as in~\cref{eq:sigestaux2} since $\xi$ can be seen as a concatenation of $m'$ curves as in the previous step, namely each one of these curves contains at least one closed curve as in~\cref{eq:sigestaux2}. Any $m'$ satisfying the above inequality is clearly a lower bound of $m$, thereby
        \begin{equation*}
            m\ge\left\lfloor\frac k{B_2}\right\rfloor\ge\frac k{B_2}-1\qquad\text{and}\qquad n<B_2,
        \end{equation*}
        where $\lfloor\cdot\rfloor$ denotes the floor function. From this and~\cref{eq:sigestaux6} we get
        \begin{equation}\label{eq:sigestaux7}
            \int_0^T\sigma_c\left(\xi,\dot\xi\right)d\tau\ge A_1\frac k{B_2}-A_1-B_1B_2.
        \end{equation}
        Finally let $\xi$ be as in~\cref{eq:sigestaux.1}, it is apparent that
        \begin{equation}\label{eq:sigestaux8}
            \xi\coloneqq(\gamma_1\circ\eta_1)*\xi'*(\gamma_k\circ\eta_k),
        \end{equation}
        where $\xi'$ is as in~\cref{eq:sigestaux4}. Notice that under our current assumptions $k\ge T$, therefore~\cref{eq:sigestaux7,eq:sigestaux8} yield
        \begin{equation}\label{eq:sigestaux9}
            \int_0^T\sigma_c\left(\xi,\dot\xi\right)d\tau\ge A_1\frac{k-2}{B_2}-A_1-B_1(B_2+2)\ge A_1\frac{T-2}{B_2}-A_1-B_1(B_2+2).
        \end{equation}
        Setting
        \begin{equation*}
            A\coloneqq\frac{A_1}{B_2}\qquad\text{and}\qquad B\coloneqq\frac2{B_2}+A_1+B_1(B_2+2),
        \end{equation*}
        \cref{eq:sigestaux9} proves~\cref{eq:sigestaux.2}.
    \end{proof}

    The next step differs from the previous one by the direction of the curve $\xi$, see the difference between the conditions on the $\eta_i$.

    \begin{lem}\label{sigestauxbis}
        Let $\xi:[0,T]\to\Gamma$ be a curve and $\{I_i\}$ an at most countable collection of open disjoint intervals with $\bigcup_i\ovI_i=[0,T]$ such that
        \begin{equation*}
            \xi\left(\ovI_i\right)\subseteq\gamma_i([0,1]),\qquad\text{for each index }i,
        \end{equation*}
        where $\gamma_i$ is an arc of the network. We set $\eta_i\coloneqq\gamma_i^{-1}\circ\xi|_{\ovI_i}$. If, for every index $i$, $\gamma_i((0,1))\cap\Acal_\Gamma=\emptyset$ and $|\dot\eta_i|=1$ a.e., then there exist two positive constants $A$ and $B$ independent of $\xi$ such that
        \begin{equation}\label{eq:sigestauxbis.1}
            AT-B\le\int_0^T\sigma_c\left(\xi,\dot\xi\right)d\tau.
        \end{equation}
    \end{lem}
    \begin{proof}
        We preliminarily assume that $\xi=\gamma\circ\eta$ with $|\dot\eta|=1$ a.e.\ for some $\gamma\in\Ecal$. We also assume, possibly replacing $\gamma$ with $\widetilde\gamma$, that
        \begin{equation}\label{eq:sigestauxbis1}
            \eta(T)\ge\eta(0).
        \end{equation}
        We will show that
        \begin{equation}\label{eq:sigestauxbis2}
            \int_0^T\sigma_c\left(\xi,\dot\xi\right)d\tau\ge\int_0^{\ovT}\sigma_c\left(\ovxi,\dot\ovxi\right)d\tau+CT_0,
        \end{equation}
        where $C$ is a suitable positive constant, $T_0\ge0$, $\ovT\coloneqq T-T_0\ge0$ and $\ovxi\coloneqq\gamma\circ\oveta$ for a curve $\oveta:\left[0,\ovT\right]\to[0,1]$ with $\dot\oveta=1$ a.e.. If $\dot\eta=1$ a.e.\ \cref{eq:sigestauxbis2} is trivial, therefore we assume that this is not the case. We will prove~\cref{eq:sigestauxbis2} proceeding by approximation.\\
        It follows from~\cref{eq:sigestauxbis1} that both
        \begin{equation*}
            E^+\coloneqq\{t\in[0,T]:\dot\eta(t)=1\}\qquad\text{and}\qquad E^-\coloneqq\{t\in[0,T]:\dot\eta(t)=-1\}
        \end{equation*}
        have positive measure. Known properties of the Lebesgue measure yield that, for each $\varepsilon>0$, there are a compact $K_\varepsilon$ and an open set $V_\varepsilon$ such that $K_\varepsilon\subseteq E^+\subseteq V_\varepsilon$ and $|V_\varepsilon\setminus K_\varepsilon|<\varepsilon$. $V_\varepsilon$ is a union of open intervals which cover the compact set $K_\varepsilon$, thereby there is an open set $V'_\varepsilon\subseteq V_\varepsilon$ made up by the union of a finite number of open intervals covering $K_\varepsilon$. Setting $E^+_\varepsilon\coloneqq V_\varepsilon'\cap[0,T]$ and $E_\varepsilon^-\coloneqq[0,T]\setminus E_\varepsilon^+$ it is apparent that they are both the finite union of some intervals and
        \begin{equation*}
            \left|\left|E^+\right|-\left|E^+_\varepsilon\right|\right|<\varepsilon,\qquad\left|\left|E^-\right|-\left|E^-_\varepsilon\right|\right|<\varepsilon,\qquad\text{for every }\varepsilon>0.
        \end{equation*}
        It is then easy to see, denoting with $\chi_E$ the characteristic function of a set $E$, that
        \begin{equation*}
            \dot\eta_\varepsilon(t)\coloneqq\chi_{E^+_\varepsilon}(t)-\chi_{E^-_\varepsilon}(t),\qquad\eta_\varepsilon(t)\coloneqq\int_0^t\dot\eta_\varepsilon(\tau)d\tau,\qquad\text{for }t\in[0,T],
        \end{equation*}
        converge a.e.\ to $\dot\eta$ and $\eta$ in $\Rds$, respectively, as $\varepsilon\to0^+$. Moreover, $\eta_\varepsilon$ is piecewise $C^1$ and $E^+_\varepsilon$, $E^-_\varepsilon$ are nonempty for $\varepsilon$ small enough. We extend by continuity $\sigma_{\gamma,c}^\pm$ on $[-1,2]$, thus, for $\varepsilon>0$ small enough, $\sigma_{\gamma,c}^\pm\circ\eta_\varepsilon$ are well-defined and, as a consequence of~\ref{condsuplin}, \cref{sigcont} and the dominated convergence Theorem,
        \begin{equation}\label{eq:sigestauxbis3}
            \begin{aligned}
                \int_0^T\sigma_c\left(\xi,\dot\xi\right)d\tau=&\int_0^T\left(\sigma_{\gamma,c}^+(\eta(\tau))\chi_{E^+}(\tau)-\sigma_{\gamma,c}^-(\eta(\tau))\chi_{E^-}(\tau)\right)d\tau\\
                =&\lim_{\varepsilon\to0^+}\int_0^T\left(\sigma_{\gamma,c}^+(\eta_\varepsilon(\tau))\chi_{E^+_\varepsilon}(\tau)-\sigma_{\gamma,c}^-(\eta_\varepsilon(\tau))\chi_{E^-_\varepsilon}(\tau)\right)d\tau.
            \end{aligned}
        \end{equation}
        For $\varepsilon$ small enough $E^+_\varepsilon$ and $E^-_\varepsilon$ are the finite union of intervals, hence we can select three points $t_1,t_2,t_3\in[0,T]$ such that $t_1<t_3$, $t_2=\dfrac{t_3+t_1}2$ and either $\dot\eta_\varepsilon=1$ on $(t_1,t_2)$ and $\dot\eta_\varepsilon=-1$ on $(t_2,t_3)$ or $\dot\eta_\varepsilon=-1$ on $(t_1,t_2)$ and $\dot\eta_\varepsilon=1$ on $(t_2,t_3)$. In particular
        \begin{equation*}
            \eta_\varepsilon(t)=\eta_\varepsilon(t_3+t_1-t),\qquad\text{for any }t\in[t_1,t_2].
        \end{equation*}
        For simplification purposes, let us say that $\dot\eta_\varepsilon=1$ on $(t_1,t_2)$. The other case is treated analogously. For $\varepsilon>0$ small enough we get
        \begin{equation*}
            \int_{t_1}^{t_2}\sigma_{\gamma,c}^+(\eta_\varepsilon(\tau))d\tau-\int_{t_2}^{t_3}\sigma_{\gamma,c}^-(\eta_\varepsilon(\tau))d\tau=\int_{t_1}^{t_2}\left(\sigma_{\gamma,c}^+(\eta_\varepsilon(\tau))-\sigma_{\gamma,c}^-(\eta_\varepsilon(t_3+t_1-\tau))\right)d\tau,
        \end{equation*}
        therefore, if we define
        \begin{equation*}
            C\coloneqq\min\frac{\left(\sigma_{\gamma,c}^+(s)-\sigma_{\gamma,c}^-(s)\right)}2,
        \end{equation*}
        where the minimum is taken over the $s\in[0,1]$ and $\gamma\in\Ecal$ such that $\gamma((0,1))\cap\Acal_\Gamma=\emptyset$, we have thanks to \cref{arcaubry} that $C>0$ and, for $\varepsilon>0$ small enough,
        \begin{equation*}
            \int_{t_1}^{t_2}\sigma_{\gamma,c}^+(\eta_\varepsilon(\tau))d\tau-\int_{t_2}^{t_3}\sigma_{\gamma,c}^-(\eta_\varepsilon(\tau))d\tau\ge C(t_3-t_1).
        \end{equation*}
        Moreover, if $\eta'_\varepsilon\coloneqq\eta_\varepsilon|_{[0,t_1]}*\eta_\varepsilon|_{[t_3,T]}$, we obtain that
        \begin{multline*}
            \int_0^T\left(\sigma_{\gamma,c}^+(\eta_\varepsilon(\tau))\chi_{E_\varepsilon^+}(\tau)-\sigma_{\gamma,c}^-(\eta_\varepsilon(\tau))\chi_{E_\varepsilon^-}(\tau)\right)d\tau\\
            \begin{aligned}
                \ge&\int_0^{t_1}\left(\sigma_{\gamma,c}^+\left(\eta_\varepsilon'(\tau)\right)\chi_{E_\varepsilon^+}(\tau)-\sigma_{\gamma,c}^-\left(\eta_\varepsilon'(\tau)\right)\chi_{E_\varepsilon^-}(\tau)\right)d\tau\\
                &+\int_{t_3}^T\left(\sigma_{\gamma,c}^+\left(\eta_\varepsilon'(\tau-t_3+t_1)\right)\chi_{E_\varepsilon^+}(\tau)-\sigma_{\gamma,c}^-\left(\eta_\varepsilon'(\tau-t_3+t_1)\right)\chi_{E_\varepsilon^-}(\tau)\right)d\tau+C(t_3-t_1).
            \end{aligned}
        \end{multline*}
        Iterating the previous step a finite number of times we get by~\cref{eq:sigestauxbis1} that, for $\varepsilon>0$ small enough, there exist $T_\varepsilon\ge0$, $\ovT_\varepsilon\coloneqq T-T_\varepsilon\ge0$ and a curve $\oveta_\varepsilon:\left[0,\ovT_\varepsilon\right]\to[0,1]$ such that $\dot\oveta_\varepsilon=1$ a.e.\ and
        \begin{equation}\label{eq:sigestauxbis4}
            \int_0^T\left(\sigma_{\gamma,c}^+(\eta_\varepsilon(\tau))\chi_{E^+_\varepsilon}(\tau)-\sigma_{\gamma,c}^-(\eta_\varepsilon(\tau))\chi_{E^-_\varepsilon}(\tau)\right)d\tau\ge\int_0^{\ovT_\varepsilon}\sigma_{\gamma,c}^+(\oveta_\varepsilon(\tau))d\tau+CT_\varepsilon.
        \end{equation}
        For $\varepsilon>0$ small enough $\{T_\varepsilon\}\subset[0,T]$ and $\{\oveta_\varepsilon\}$ is an equibounded and equicontinuous collection of curves with $\dot\oveta_\varepsilon=1$ a.e., therefore the Arzelà--Ascoli Theorem yields that there is an infinitesimal subsequence $\{\varepsilon_n\}_{n\in\Nds}$, the times $T_0\ge0$, $\ovT\coloneqq T-T_0\ge0$ and a curve $\oveta:\left[0,\ovT\right]\to[0,1]$ with $\dot\oveta=1$ a.e.\ such that, as $n\to\infty$, $T_{\varepsilon_n}\to T_0$, $\ovT_{\varepsilon_n}\to\ovT$ and $\oveta_{\varepsilon_n}$ uniformly converges to $\oveta$. Finally~\cref{eq:sigestauxbis3,eq:sigestauxbis4} proves~\cref{eq:sigestauxbis2}.\\
        Now let $\xi$ be as in the statement and define $\xi_i\coloneqq\gamma_i\circ\eta_i$. Repeating the previous analysis we obtain for each $\xi_i$, possibly replacing $\gamma_i$ with $\wtgamma_i$, the nonnegative number $\ovT_i$, the curve $\oveta_i:\left[0,\ovT_i\right]\to[0,1]$ with $\oveta_i=1$ a.e.\ and $\ovxi_i\coloneqq\gamma_i\circ\oveta_i$ such that
        \begin{equation}\label{eq:sigestauxbis5}
            T_0\coloneqq T-\sum_i\ovT_i\ge0,
        \end{equation}
        \begin{equation*}
            \int_0^T\sigma_c\left(\xi,\dot\xi\right)d\tau\ge\sum_i\int_0^{\ovT_i}\sigma_c\left(\ovxi_i,\dot\ovxi_i\right)d\tau+CT_0.
        \end{equation*}
        We point out that each $\oveta_i$ whose support is not a point, except at most two, links the extremes of the arc $\gamma_i$ and has speed equal to 1. Thereby for each $\ovT_i$, except at most two, $\ovT_i\in\{0,1\}$. We notice that, due to~\cref{eq:sigestauxbis5}, there are only a finite number of positive times $\ovT_i$. It then apparent that the $\ovxi_i$ with $\ovT_i>0$ can be concatenated into a curve $\ovxi:\left[0,\ovT\right]\to[0,1]$, where $\ovT\coloneqq\sum_i\ovT_i$, therefore
        \begin{equation*}
            \int_0^T\sigma_c\left(\xi,\dot\xi\right)d\tau\ge\int_0^{\ovT}\sigma_c\left(\ovxi,\dot\ovxi\right)d\tau+CT_0.
        \end{equation*}
        $\ovxi$ satisfies the hypotheses of \cref{sigestaux}, thus
        \begin{equation*}
            \int_0^T\sigma_c\left(\xi,\dot\xi\right)d\tau\ge A'\ovT-B+CT_0,
        \end{equation*}
        for some positive constants $A'$, $B$. Setting $A\coloneqq A'\wedge C$, this proves~\cref{eq:sigestauxbis.1}.
    \end{proof}

    \begin{proof}[Proof of \cref{sigest}]
        If $\dot\xi=0$ a.e.\ \cref{eq:sigest.2} is trivial, hence we assume that $\xi$ is not constant. Exploiting \cref{curveascomp} there is an at most countable collection of open disjoint intervals $\{I_i\}$ with $\bigcup_i\ovI_i=[0,T]$ such that
        \begin{equation*}
            \xi\left(\ovI_i\right)\subseteq\gamma_i([0,1]),\qquad\text{for each index }i,
        \end{equation*}
        where $\gamma_i$ is an arc of the network such that $\gamma_i((0,1))\cap\Acal_\Gamma=\emptyset$. $\xi$ is a nonconstant curve, thereby, possibly combining some intervals, we can assume that each $\eta_i\coloneqq\gamma_i^{-1}\circ\xi|_{\ovI_i}$ is a nonconstant curve from $\ovI_i$ into $[0,1]$. We know from \cref{repconstspeed} that for every index $i$ there is a curve $\eta_i':[0,T_i']\to[0,1]$ with $|\dot\eta_i'|=1$ a.e.\ and an absolutely continuous function $\psi_i:\ovI_i\to[0,T_i']$ such that $\eta_i=\eta_i'\circ\psi_i$. We have by \cref{reparac} that
        \begin{equation*}
            \int_{\ovI_i}|\dot\eta_i(\tau)|d\tau=\int_{\ovI_i}\left|\dot\eta_i'(\psi_i(\tau))\right|\left|\dot\psi_i(\tau)\right|d\tau=T_i'.
        \end{equation*}
        Setting $T'\coloneqq\sum_i T_i'$, $C_1\coloneqq\min\limits_{\gamma\in\Ecal,s\in[0,1]}|\dot\gamma(s)|_2>0$ and $C_2\coloneqq\max\limits_{\gamma\in\Ecal,s\in[0,1]}|\dot\gamma(s)|_2$, it then follows that
        \begin{equation}\label{eq:sigest1}
            C_1T'\le\int_0^T\left|\dot\xi(\tau)\right|_2d\tau=\sum_i\int_{\ovI_i}|\dot\gamma_i(\eta_i(\tau))|_2|\dot\eta_i(\tau)|d\tau\le C_2T'.
        \end{equation}
        In particular $T'$ is bounded. Next we notice that each $I_i$ can be seen as the open interval $(t_i,t_i+T_i)$, for suitable $t_i$ and $T_i$, then we define for all the indices $i$
        \begin{equation*}
            t_i'\coloneqq\sum_{j:t_j<t_i}T_j',\qquad I'_i\coloneqq\left(t_i',t_i'+T_i'\right).
        \end{equation*}
        These intervals are clearly disjoint and $\bigcup_i\overline{I'_i}=[0,T']$. Finally the curve $\xi':[0,T']\to\Gamma$ defined by
        \begin{equation*}
            \xi'|_{\overline{I'_i}}(t)\coloneqq\gamma_i\circ\eta'_i\left(t-t_i'\right),\qquad\text{for all index }i,
        \end{equation*}
        is a reparametrization of $\xi$ and satisfies the assumptions of \cref{sigestauxbis}, thus \cref{eq:sigest1,repsigma} yield
        \begin{equation*}
            \frac{A'}{C_2}\int_0^T\left|\dot\xi(\tau)\right|_2d\tau-B\le A'T'-B\le\int_0^{T'}\sigma_c\left(\xi',\dot\xi'\right)d\tau=\int_0^T\sigma_c\left(\xi,\dot\xi\right)d\tau,
        \end{equation*}
        where $A'$ and $B$ are positive constant independent of $\xi$. This proves~\cref{eq:sigest.2}.
    \end{proof}

    We can now proceed to the proof of \cref{liveoptcurve}.

    \begin{proof}[Proof of \cref{liveoptcurve}]
        We fix $(x,t)\in\Gamma\times\Rds^+$ and an optimal curve $\xi$ for $(\Scal(t)\phi)(x)$. We define the constant
        \begin{equation*}
            l_0\coloneqq\max\left\{c_x:x\in\Vbf\setminus\wtAcal_\Gamma\right\}\vee\max\{a_\gamma:\gamma\in\Ecal,\,\gamma((0,1))\cap\Acal_\Gamma=\emptyset\},
        \end{equation*}
        then, by \cref{arcaubry,eq:loclagmin},
        \begin{equation}\label{eq:liveoptcurve1}
            L(x,0)+c\ge c-l_0>0,\qquad\text{for any }x\in\Gamma\setminus\wtAcal_\Gamma.
        \end{equation}
        We break the argument according to the sign of $L(x,q)+l_0$. If
        \begin{equation*}
            L(x,q)+l_0\ge0,\qquad\text{for every $(x,q)\in T\Gamma$ with }x\in\Gamma\setminus\wtAcal_\Gamma,
        \end{equation*}
        and $\xi$ is disjoint from $\wtAcal_\Gamma$ in $[0,t]$, then by~\cref{eq:asbound}
        \begin{equation*}
            \min_{x\in\Gamma}\phi(x)+(c-l_0)t\le\phi(\xi(0))+\int_0^t\left(L\left(\xi,\dot\xi\right)+c\right)d\tau\le\max_{x\in\Gamma}u(x)+\alpha,
        \end{equation*}
        which implies that
        \begin{equation*}
            t\le\frac{\max\limits_{x\in\Gamma}u(x)-\min\limits_{x\in\Gamma}\phi(x)+\alpha}{c-l_0}\eqqcolon T_\phi,
        \end{equation*}
        yielding the assertion.\\
        Next, we assume that $L(x,q)+l_0<0$ for some $(x,q)\in T\Gamma$ with $x\in\Gamma\setminus\wtAcal_\Gamma$, and set
        \begin{equation*}
            r_\delta\coloneqq\min\left\{|q|_2:\text{$(x,q)\in T\Gamma$ for some $x\in\Gamma\setminus\wtAcal_\Gamma$, $L(x,q)\le-l_0-\delta$}\right\}
        \end{equation*}
        and $C_\delta\coloneqq c-l_0-\delta$, where $\delta>0$ is such that $r_\delta>0$. We point out that the existence of such $\delta$ is a consequence of~\cref{eq:liveoptcurve1}. It follows that
        \begin{equation}\label{eq:liveoptcurve2}
            L(x,q)+c>C_\delta>0,\qquad\text{for any $(x,q)\in T\Gamma$ with }|q|_2<r_\delta,x\in\Gamma\setminus\wtAcal_\Gamma.
        \end{equation}
        We assume that $\xi$ is disjoint from $\wtAcal_\Gamma$ in $[0,t]$ and define the set
        \begin{equation*}
            E\coloneqq\left\{\tau\in[0,t]:\left|\dot\xi(\tau)\right|_2<r_\delta\right\},
        \end{equation*}
        then~\cref{eq:asbound}, \cref{sigest,laglbound} show that there exist two positive constants $A$ and $B$, independent of $\xi$, such that
        \begin{align*}
            A(t-|E|)r_\delta-B\le&\,A\int_0^t\left|\dot\xi(\tau)\right|_2d\tau-B\le\int_0^t\sigma_c\left(\xi,\dot\xi\right)d\tau\le\int_0^t\left(L\left(\xi,\dot\xi\right)+c\right)d\tau\\
            \le&\,\max_{x\in\Gamma}u(x)-\min_{x\in\Gamma}\phi(x)+\alpha.
        \end{align*}
        This in turn yields that there is a constant $C_1$ depending on $\phi$ such that
        \begin{equation}\label{eq:liveoptcurve3}
            t-C_1\le|E|.
        \end{equation}
        By~\cref{eq:asbound,eq:liveoptcurve2} we similarly have
        \begin{equation*}
            (t-|E|)\min_{x\in\Gamma\setminus\wtAcal_\Gamma,|q|_2\ge r_\delta}(L(x,q)+c)+C_\delta|E|\le\int_0^t\left(L\left(\xi,\dot\xi\right)+c\right)d\tau\le\max_{x\in\Gamma}u(x)-\min_{x\in\Gamma}\phi(x)+\alpha,
        \end{equation*}
        thus if we define
        \begin{equation*}
            C_2\coloneqq0\vee\left(-\min_{x\in\Gamma\setminus\wtAcal_\Gamma,|q|_2\ge r_\delta}(L(x,q)+c)\right)
        \end{equation*}
        we get
        \begin{equation}\label{eq:liveoptcurve4}
            -C_2t+(C_\delta+C_2)|E|\le\max_{x\in\Gamma}u(x)-\min_{x\in\Gamma}\phi(x)+\alpha.
        \end{equation}
        Finally we combine \cref{eq:liveoptcurve3} with~\cref{eq:liveoptcurve4} to obtain that
        \begin{equation*}
            C_\delta t-(C_\delta+C_2)C_1\le\max_{x\in\Gamma}u(x)-\min_{x\in\Gamma}\phi(x)+\alpha,
        \end{equation*}
        which proves, also in this case, that there is a constant $T_\phi$, depending only on $\phi$, such that $\xi$ is disjoint from $\wtAcal_\Gamma$ in $[0,t]$ only if $t\le T_\phi$. This concludes the proof.
    \end{proof}

    For $a>c$ it is possible to obtain an analogue of \cref{sigest}, and consequently of \cref{liveoptcurve}, with straightforward modifications. The main difference is the presence of the condition~\cref{eq:sigest.1}, which is only used in~\cref{eq:sigestaux3} to obtain the positive constant $A_1$. Such condition is not needed since, for any $a>c$,
    \begin{equation*}
        \int_0^T\sigma_a\left(\xi,\dot\xi\right)d\tau>0
    \end{equation*}
    whenever $\xi:[0,T]\to\Gamma$ is a nonconstant closed curve. More in details we have:

    \begin{lem}\label{sigestgec}
        Given $a>c$ there exist two positive constants $A$ and $B$, depending only on $a$, such that, for any curve $\xi:[0,T]\to\Gamma$,
        \begin{equation*}
            A\int_0^T\left|\dot\xi(\tau)\right|_2d\tau-B\le\int_0^T\sigma_a\left(\xi,\dot\xi\right)d\tau.
        \end{equation*}
    \end{lem}

    Arguing as in the proof of \cref{liveoptcurve}, using \cref{sigestgec} instead of \cref{sigest} and with straightforward modifications, we get the next \namecref{supliveoptcurve}.

    \begin{lem}\label{supliveoptcurve}
        Given a flux limiter $c_x$ such that $a\coloneqq\max\limits_{x\in\Vbf}c_x>c$ and a $\phi\in C(\Gamma)$, there is a $T_\phi>0$ depending only on $\phi$ and $a$ such that, for any $x\in\Gamma$, $t\ge T_\phi$ and optimal curve $\xi$ for $(\Scal(t)\phi)(x)$, $\xi([0,t])\cap\Vbf_a\ne\emptyset$.
    \end{lem}

    \printbibliography[heading=bibintoc]

\end{document}